\documentclass[12pt,reqno]{amsart}
\usepackage[utf8]{inputenc}
\usepackage[activeacute,english]{babel}

\usepackage[margin=0.9in]{geometry}
\usepackage[hidelinks]{hyperref} 

\usepackage[leqno]{amsmath}
\usepackage{enumerate}
\usepackage{amssymb}
\usepackage{mathtools}
\usepackage{physics}
\usepackage{amsthm}
\usepackage{mathrsfs}
\usepackage[nocompress]{cite}
\usepackage{verbatim}
\usepackage[alphabetic]{amsrefs}

\usepackage{needspace}

\usepackage{pgfplots}
\pgfplotsset{compat=1.15}
\usepackage{mathrsfs}
\usetikzlibrary{arrows}

\mathtoolsset{showonlyrefs,showmanualtags}  

\hypersetup{colorlinks=true, linkcolor=blue, citecolor=red}

\theoremstyle{plain}
\newtheorem{theorem}{Theorem}[section]
\newtheorem{proposition}[theorem]{Proposition}
\newtheorem{corollary}[theorem]{Corollary}
\newtheorem{lemma}[theorem]{Lemma}
\theoremstyle{remark}
\newtheorem{remark}[theorem]{Remark}

\theoremstyle{definition}
\newtheorem{definition}[theorem]{Definition}

\allowdisplaybreaks[1]

\numberwithin{equation}{section}

\makeatletter
\newcommand{\proofstep}[1]{%
  \par
  \addvspace{\medskipamount}
  \textit{#1\@addpunct{.}}\enspace\ignorespaces
}
\makeatother

\renewcommand{\expval}[1]{\langle {#1} \rangle}
\newcommand{\smabs}[1]{| {#1} |}

\begin{document}
\title[On the shape of minimizers for the periodic nonlocal perimeter in $\mathbb{R}^2$]{On the shape of minimizers for the\\periodic nonlocal perimeter in $\mathbb{R}^2$}

\author[Renzo Bruera]{Renzo Bruera}
\address{R. Bruera \textsuperscript{1}
\newline
\textsuperscript{1}
Universitat Politècnica de Catalunya, Departament de Matem\`{a}tiques, Av. Diagonal 647, 08028 Barcelona.}
\email{renzo.bruera@upc.edu}

\date{\today}
\thanks{This work is supported by the Spanish State Research Agency, through the Severo Ochoa and Mar\'{\i}a de Maeztu Program for Centers and Units of Excellence in R\&D (CEX2020-001084-M). Renzo Bruera is supported by the Spanish Ministry of Universities through the national program FPU (reference FPU20/07006), by the Spanish grants PID2021-123903NB-I00 and RED2022-134784-T funded by MCIN/AEI/10.13039/501100011033 and by ERDF “A way of making Europe”, and by the Catalan grant 2021-SGR-00087.}

\begin{abstract}

In this paper, we study planar nonlocal Delaunay sets. That is, open sets in $\mathbb{R}^2$ with constant nonlocal mean curvature that are periodic in $x_1$, and even in $x_1$ and in $x_2$. Using bifurcation analysis and fine explicit computations, we prove that every sufficiently $C^{1,\beta}$-flat nonlocal Delaunay set in $\mathbb{R}^2$ that is not a straight band is unstable with respect to volume-preserving periodic variations. 

Our results support the conjecture that, as in the local case, in the range of large areas, minimizers of the periodic nonlocal isoperimetric problem---also known as the nonlocal liquid drop problem with prescribed area between two parallel hyperplanes---are all straight bands.
\end{abstract}

\maketitle

\tableofcontents

\section{Introduction}

This paper concerns the stability of periodic, even, cylindrically symmetric sets of constant nonlocal mean curvature (NMC) in dimension $n=2$. That is, we consider sets $E\subset \mathbb{R}^2$ of the form $E=\{(z_1,z_2)\in \mathbb{R}^2 : \abs{z_2}<u(z_1)\}$ for some $2\pi$-periodic, even function $u:\mathbb{R}\to \mathbb{R}$ that satisfy, for every $x\in \partial E$,
\begin{equation}
    H_{\alpha}[E](x) = -\mathrm{PV}\int_{\mathbb{R}^2} \{\chi_E(y)-\chi_{E^{c}}(y)\}K(x-y)\dd{y} = \mathrm{constant},
\end{equation}
where $K(z) = \abs{z}^{-2-\alpha}$ with $\alpha\in (0,1)$, and $\chi_A$ denotes the characteristic function of the set $A$, and we study their stability when viewed as volume-constrained critical points of the periodic nonlocal perimeter functional,
\begin{equation}\label{eq:pnp_intro}
    \mathcal{P}_{\alpha}[E] = \int_{E\cap \Omega}\int_{E^c} K(x-y)\dd{x}\dd{y},
\end{equation}
with $\Omega = \{(z_1,z_2)\in \mathbb{R}^2 : -\pi < z_1 < \pi\}$.

This problem is the two-dimensional nonlocal analogue of a classical problem known as the liquid drop problem with prescribed volume between two parallel hyperplanes. Given a volume constraint $\omega>0$, one wishes to find the subset $\tilde{E}\subset \mathbb{R}^n$ with volume $\omega$ in $(0,\pi)\times \mathbb{R}^{n-1}$ (i.e., satisfying $\smabs{\tilde{E}\cap (0,\pi)\times \mathbb{R}^{n-1}}=\omega $) with the minimum (classical) perimeter in $(0,\pi)\times \mathbb{R}^{n-1}$, without including the perimeter contained in the hyperplanes $\{0\}\times \mathbb{R}^{n-1}$ and $\{\pi\}\times \mathbb{R}^{n-1}$---this is often called the perimeter of the \textit{free boundary} of $\tilde{E}$. By considering the even reflection of a $\tilde{E}$, one can relate this problem to the periodic isoperimetric problem, i.e., the isoperimetric problem relative to the slab $(-\pi, \pi)\times \mathbb{R}^{n-1}$ with periodic boundary conditions. 

It is known that the $2\pi$-periodic extension of the free boundary of every minimizer of the classical periodic isoperimetric problem is a $2\pi$-periodic surface of revolution with constant mean curvature. These are commonly referred to as \textit{Delaunay surfaces}, in honour of the work of Charles Delaunay, who characterized them in dimension $n=3$ \cite{Delaunay}. There are only three kinds of embedded Delaunay surfaces: straight cylinders, arrays of spheres, and unduloids, although the former two are, in fact, limiting cases of unduloids. In the literature, the name ``Delaunay surfaces'' is often reserved for this last kind. Unduloids are the surfaces of revolution whose generatrix is the curve traced by one of the foci of an ellipse that is rolled without slipping along a straight line. When the ellipse has eccentricity 0 (i.e., when it is a circumference), the surface of revolution generated is a straight cylinder, and in the limit in which the ellipse is a line segment (i.e., when its semi-minor axis has length 0), the surface of revolution generated is an array of spheres.

Still in the local setting, it has been shown (see \cite{PedrosaRitore1999} for the following claims) that, in dimensions $3\leq n \leq 8$, every minimizer of the periodic isoperimetric problem is either an array of spheres or a straight cylinder, depending on whether the prescribed volume $\omega$ is larger or smaller than some critical volume $\omega^*$. In particular, no unduloid is a minimizer for any volume constraint in dimensions $3\leq n \leq 8$. More generally, unduloids are known to be unstable in dimensions $3\leq n\leq 8$---here, by stable we mean that the second variation of the perimeter functional at the free boundary of the unduloid is nonnegative. In contrast, when $n\geq 10$, there exist Delaunay unduloids that are minimizers for certain volumes (and, as a consequence, they are stable). The problem remains open for $n=9$. 

Formulating the nonlocal analogue of the periodic isoperimetric problem is not entirely straightforward. Recall that, given a set $\tilde{E}\subset \mathbb{R}^n$, its nonlocal perimeter is defined by
\begin{equation}
    \mathrm{Per}_{\alpha}[\tilde{E}] = \int_{\tilde{E}}\int_{\tilde{E}^c} K(x-y)\dd{x}\dd{y},
\end{equation}
with $K(z)=\abs{z}^{-n-\alpha}$ and $\alpha \in (0,1)$. Whereas in the local problem the periodicity of $\tilde{E}$ appears only in the boundary conditions, the nonlocal perimeter of $\tilde{E}$ in the slab $\Omega := (-\pi,\pi)\times \mathbb{R}^{n-1}$ depends both on $\tilde{E}\cap \Omega$ and $\tilde{E}\cap \Omega^c$. It turns out (see also \cite{AlcoverBruera} for a detailed discussion on this fact) that the appropriate notion of nonlocal perimeter for sets that are $2\pi$-periodic in the $x_1$-direction is given by the functional $\mathcal{P}_{\alpha}$ defined in \eqref{eq:pnp_intro}. 

The periodic nonlocal isoperimetric problem is then formulated as follows. In any dimension $n\geq 2$, and for a given volume constraint $\omega>0$, we wish to minimize the functional $\mathcal{P}_{\alpha}$ among all sets $F\subset \mathbb{R}^n$ that are $2\pi$-periodic in the coordinate $x_1$ and have volume $\omega$ within the slab $\Omega$, i.e., $\abs{F\cap \Omega}=\omega$. In \cite{CCM1}, Cabré, Csató, and Mas showed that volume-constrained minimizers of $\mathcal{P}_{\alpha}$ exist for every $n\geq 2$ and every $\omega >0$. Moreover, they proved that every minimizer must be even and cylindrically symmetric---i.e., of the form
\begin{equation}
    E = \{(x_1,x')\in \mathbb{R}^n\simeq \mathbb{R}\times \mathbb{R}^{n-1} : \abs{x'}<u(x_1)\}
\end{equation}
for some $2\pi$-periodic, even function $u:\mathbb{R}\to \mathbb{R}$---, and that, provided that the minimizer is sufficiently regular, it must have constant NMC. Notice that the nonlocal problem is interesting even in dimension $n=2$, unlike in the local setting, where all curves of constant curvature are either straight lines or circular arcs. Surfaces of revolution that are periodic and even in $x_1$, and have constant NMC are called \textit{nonlocal Delaunay surfaces}, by analogy with their local counterparts.

The existence of constant-NMC surfaces of revolution had also been previously studied by Cabré, Solà-Morales, Fall, and Weth in \cite{CNLMCDelCyl} for $n=2$, and then by Cabré, Fall, and Weth in \cite{CabreFallWeth2018} for $n\geq 2$. By applying perturbation techniques (essentially, the implicit function theorem), they showed the existence of a continuous family of periodic constant-NMC surfaces of revolution bifurcating from a straight cylinder, i.e., from a set of the form $\{(x_1,x')\in \mathbb{R}^n:\abs{x'}<R\}$, for a particular value of $R>0$. The sets obtained via these perturbation techniques are called \textit{near-cylinders}. They are constant-NMC, $2\pi$-periodic, even surfaces of revolution (i.e., nonlocal Delaunay surfaces, as defined before) that are close to a straight cylinder. Notice that when $n=2$, straight cylinders are, in fact, straight bands in $\mathbb{R}^2$, and thus near-cylinders should be called near-bands. Even so, we will call them cylinders for consistency with the usual term in $n\geq 3$.

It can be shown (see Remark \ref{remark_constant NMC_critical_point} in Section \ref{section:on_the_stab_of_periodic_NMC_surf_of_rev}) that every $2\pi$-periodic set with constant NMC is a volume-constrained critical point of $\mathcal{P}_{\alpha}$. It is then natural to ask whether the near-cylinders described in the previous paragraph are volume-constrained minimizers of $\mathcal{P}_{\alpha}$, or whether they are, at least, stable.

In this paper, we show that, in dimension $2$, all non-straight near-cylinders are unstable with respect to even, cylindrically symmetric volume-preserving periodic variations. From this, we deduce that, given a volume constraint, if a minimizer of $\mathcal{P}_{\alpha}$ in dimension $n=2$ is sufficiently $C^{1,\beta}$-flat, then it must be a straight cylinder. Now, in order to explain this further, we make the following remark, to which we will refer later on.

\begin{remark}\label{remark:bifurcation}
    Let us start by describing the stability or instability of straight cylinders. First, we note that there is a continuum of them: for every radius $R>0$, the straight band of amplitude $2R$ (to which we will associate the generatrix function $u\equiv R$) satisfies the equation
    \begin{equation}\label{eq:NMC_eq_intro}
        H_{\alpha}(u) - h_R = 0,
    \end{equation}
    where $h_R := H_{\alpha}(R)$ is its nonlocal mean curvature. Its linearization with respect to $u$ is an integro-differential operator $D_u H_{\alpha}(R)v$ acting on $2\pi$-periodic, even functions $v=v(t)$. 

    In Section \ref{section:on_the_stab_of_periodic_NMC_surf_of_rev}, we will see that the stability of straight cylinders is determined by the sign of the second eigenvalue of $D_u H_{\alpha}(R)$ (the first eigenvalue is always negative, and it corresponds to non-volume-preserving variations, which are not admissible). We will show that there exists a critical radius $R_1$ such that the straight cylinder of radius $R>0$ is strictly stable if $R>R_1$, it is unstable if $R<R_1$, and it is degenerate stable (i.e., semi-stable) if $R=R_1$. 

    When $R>R_1$, since the linearized operator $D_uH_{\alpha}(R)$ is non-degenerate, the implicit function theorem applied to the equation \eqref{eq:NMC_eq_intro}, which allows us to find the solutions to \eqref{eq:NMC_eq_intro} as a function of $R$, will yield the trivial curve of solutions $u \equiv R$. In other words, there are no near-cylinders bifurcating from straight cylinders of radii $R>R_1$.

    On the other hand, we will see that there exists a decreasing sequence $R_1>R_2>\cdots >R_k \downarrow 0$ of bifurcating radii. Each one corresponds to the radius $R$ at which one of the eigenvalues of $D_uH_{\alpha}(R)$ vanishes. For $k\geq 2$, the straight cylinder of radius $R_k$ is unstable (since $R_k<R_1$ for $k\geq 2$), and, therefore, the branch of near-cylinders bifurcating from $R_k$ will be unstable too. 

    When $R=R_1$, we recall that the straight cylinder of radius $R_1$ is degenerate stable (since the second eigenvalue of $D_u H_{\alpha}(R)$ vanishes). As a consequence, the branch of near-cylinders bifurcating from $R_1$ could be either stable or unstable. By explicit computations, we will see that they are, in fact, unstable. This is the main contribution of this paper.
\end{remark}
Since some of our results will concern the limit $\alpha\uparrow 1$, let us comment on the behaviour of the bifurcating radii $R_m$ in the local case, $\alpha=1$.
\begin{remark}\label{remark:radius_alpha_1}
Notice that, in the local case, all straight bands are sets of constant curvature equal to 0. As a consequence, the linearized or Jacobi operator is the operator $-\Delta$ with periodic boundary conditions. This operator is positive definite, and thus all straight bands are stable. Therefore, there cannot exist any bifurcating radii when $\alpha=1$, i.e., we have $R_1(\alpha=1)=0$.
\end{remark}

\subsection{Main results} First, we extend the perturbative result of \cite{CNLMCDelCyl} by showing that there exists a \textit{sequence} of smooth families of constant-NMC near-cylinders in $\mathbb{R}^2$ that converge to the set $\{0\}\times \mathbb{R}$ as $m\to +\infty$. To state the theorem, we need to introduce some notation. Throughout the whole text, we will let $\alpha\in (0,1)$ and $\beta\in (\alpha, 1)$. We denote by $C^{1,\beta}_{even}(\mathbb{T}^1)$ the space of $2\pi$-periodic, even functions of class $C^{1,\beta}$ endowed with the usual norm. That is,
\begin{equation}
    C^{1,\beta}_{even}(\mathbb{T}^1) := \{ u \in C^{1,\beta}(\mathbb{R}) : \ u \text{ is }2\pi\text{-periodic and even} \},
\end{equation}
with $\norm{u}_{C^{1,\beta}_{even}(\mathbb{T}^1)} := \norm{u}_{C^{1,\beta}([-\pi,\pi])}$. The spaces $C^{\beta- \alpha}_{even}(\mathbb{T}^1)$ and $L^{2}_{even}(\mathbb{T}^1)$ are defined similarly. 

Slightly abusing notation, for a cylindrically symmetric set with generatrix function $u$---that is, a set of the form $E=\{(z_1,z_2) \in \mathbb{R}^2 : \abs{z_2}<u(z_1)\}$ for some positive function $u:\mathbb{R}\to \mathbb{R}$---we denote by
\begin{equation}
    H_{\alpha}(u):=H_{\alpha}[E](s, u(s)),
\end{equation}
the NMC of the set $E$ at the point $(s,u(s))\in \partial E$. In this way, when restricted to the class of cylindrically symmetric sets, the NMC operator can be seen as an operator acting on functions. 

When $u$ is constant, i.e., when $u\equiv R$, with $R>0$, we will see in Section \ref{section:a_seq_of_bif_delaunay} that $H_{\alpha}(R)$ is constant and positive. For $R>0$, we denote by $u_R$ the constant function equal to $R$, i.e., $u_R \equiv R$, and by $h_R := H_{\alpha}(u_R)=H_{\alpha}(R)$ the NMC of the straight cylinder of radius $R$. 

Our first main result is the following. Using bifurcation analysis, we show the existence and local uniqueness of a sequence of bifurcation branches of constant-NMC near-cylinders in $\mathbb{R}^2$, thereby extending the results of \cite{CNLMCDelCyl}, that proved existence of the first branch only.
\begin{theorem}\label{theorem_seq_bif}
Let $n=2$, $\alpha\in (0,1)$ and $\alpha< \beta < \min\{1, 2\alpha+\frac{1}{2}\}$. There exists $R_1>0$, depending only on $\alpha$, such that, for every $m\geq 1$, denoting $R_m=R_1/m$, the pair $(R_m,u_{R_m})$ is a bifurcation point of the equation 
\begin{equation}\label{equation_H(Ru)}
    \mathcal{H}(R,u):=H_\alpha(u)-h_R=0, \quad (R,u)\in \mathbb{R}\times C^{1,\beta}_{even}(\mathbb{T}^1).
\end{equation}

As a consequence, for every $m\geq 1$ there exists $\nu_m>0$, depending only on $\alpha$, $\beta$, and $m$, and a smooth curve of nontrivial solutions to \eqref{equation_H(Ru)},
\begin{equation}\label{eq::expression_solutions}
    \left\{(\gamma_m(a)R_m,\  w_m(a)) : a\in (-\nu_m,\nu_m)\right\}\subset \mathbb{R}\times C^{1,\beta}_{even}(\mathbb{T}^1),
\end{equation}
with every $w_m(a)$ of the form
\begin{equation}
  w_m(a)(\cdot) = \gamma_m(a)R_m + a (\cos(m\cdot)+v_m(a)(\cdot)),
\end{equation}
with $v_m(a)$ orthogonal to $\cos(m\cdot)$ in $L^2_{even}(\mathbb{T}^1)$, and  $(\gamma_m(a),v_m(a))\xrightarrow[a\to 0]{} (1,0)$ in $\mathbb{R}\times C^{1,\beta}_{even}(\mathbb{T}^1)$.

In particular, for every $m\geq 1$ and every $a\in (- \nu_m, \nu_m)$, the sets
    \begin{equation}\label{eq::expression_near_cylinders}
        E_m(a) = \{(z_1,z_2) \in \mathbb{R}^2 : \abs{z_2}<w_m(a)\}
    \end{equation}
have each constant nonlocal mean curvature equal to $h_{\gamma_m(a) R_m}$.

Moreover, there exists a neighbourhood $\mathcal{U}$ of the set $\{(R,u_R):R>0\}$ in $\mathbb{R}\times C^{1,\beta}_{even}(\mathbb{T}^1)$ such that every solution to \eqref{equation_H(Ru)} in $\mathcal{U}$ is either a trivial solution, or one of the bifurcated solutions $(\gamma_m(a)R_m, w_m(a))$ for some $m\geq 1$ and some $a\in (-\nu_m, \nu_m)$. In other words,
\begin{equation}
    \mathcal{H}^{-1}(0)\cap \mathcal{U} = \{(R,u_R):R>0\} \cup  \left(\bigcup_{m\geq 1}\left\{(\gamma_m(a)R_m,\  w_m(a)) : a\in (-\nu_m,\nu_m)\right\}\right).
\end{equation}
\end{theorem}
\begin{remark}\label{remark:dependence_beta}
We note that the assumption that $\beta<2 \alpha + \frac{1}{2}$ is merely technical and it entails no real loss of generality. Indeed, the choice of $\beta$ only possibly affects the size of the open sets $(-\nu_m, \nu_m)$ around $a=0$ in which each of the bifurcation branches is defined. In other words, there could be a loss of regularity of the functions $w_m(a)$ as $\abs{a}$ is increased. However, by the local uniqueness of solutions, each of the branches $\{ w_m(a) : a \in (- \nu_m, \nu_m)\}$ and $\{\tilde{w}_m(a) : a\in (- \tilde{\nu}_m, \tilde{\nu}_m')\}$ found, respectively, for any two choices $\beta$ and $\tilde{\beta}$ must coincide in the common domain of definition, i.e., $w_m(a) = \tilde{w}_m(a)$ for all $a\in(- \nu_m, \nu_m)\cap (- \tilde{\nu}_m, \tilde{\nu}_m)$.
\end{remark}

With Theorem \ref{theorem_seq_bif} in hand, we can state our second main result, which is the main novelty of this paper. It concerns the stability of both trivial and nontrivial solutions to \eqref{equation_H(Ru)} when viewed as volume-constrained critical points of the functional $\mathcal{P}_{\alpha}$. We introduce here the notion of stability, which will be treated more in depth in Section \ref{section:on_the_stab_of_periodic_NMC_surf_of_rev}. As we will explain later, every even, cylindrically symmetric $2\pi$-periodic set $E$ with constant NMC is a volume-constrained critical point of the periodic nonlocal perimeter, $\mathcal{P}_{\alpha}$. That is, we have
\begin{equation}
    \derivative{t}\mathcal{P}_{\alpha}[E_t]\bigg\vert_{t=0} = 0
\end{equation}
for all even, cylindrically symmetric, volume-preserving periodic variations $E_t$ of $E$---i.e., smooth deformations of the set $E$ that preserve its symmetry, its periodicity, and its volume within the slab, $\abs{E\cap \Omega}$, where, as usual, we have denoted $\Omega = \{(z_1,z_2)\in \mathbb{R}^2 : -\pi < z_1 < \pi\}$. We say that $E$ is \emph{stable} for $\mathcal{P}_{\alpha}$ if
\begin{equation}
    \derivative[2]{t}\mathcal{P}_{\alpha}[E_t] \bigg\vert_{t=0} \geq 0
\end{equation}
for all even, cylindrically symmetric, volume-preserving periodic variations $E_t$ of $E$.

Our second main result characterizes the minimum radius for stability of straight cylinders and, more importantly, it shows the instability of all near-cylinders except for the straight cylinders of radius $R>R_1$; see Remark \ref{remark:bifurcation} for the global bifurcation picture.
\needspace{4\baselineskip}
\begin{theorem}\label{theorem_stability_near_cylinders}
Let $\alpha\in (0,1)$ and $\alpha<\beta<\min\{1, 2\alpha + \frac{1}{2}\}$. In the setting and the notation of Theorem \ref{theorem_seq_bif}:
\begin{enumerate}
    \item Every straight cylinder of radius $R\geq R_1$ is stable for $\mathcal{P}_{\alpha}$, and every straight cylinder of radius $0<R<R_1$ is unstable. Here, $R_1$ is the radius in Theorem~\ref{theorem_seq_bif}, which depends only on $\alpha$.

    Moreover, 
    \begin{equation}\label{eq:asymptotic_R1_alpha_1}
     R_1=O(\sqrt{1-\alpha}) \quad \text{as} \quad \alpha\uparrow 1 .
    \end{equation}

    \item There exists $\alpha_0 \in (0,1)$ such that, if $\alpha \in [\alpha_0,1)$, for all $\abs{a}>0$ small enough, depending only on $\alpha$ and $\beta$, the sets
    \begin{equation}
        E_1(a) = \{(z_1,z_2) \in \mathbb{R}^2 : \abs{z_2}<w_1(a)\}
    \end{equation}
    are unstable for $\mathcal{P}_{\alpha}$.

    \item For every $m\geq 2$, for all $\abs{a}$ small enough, depending only on $\alpha$, $\beta$, and $m$, the sets
    \begin{equation}
        E_m(a)= \{(z_1,z_2) \in \mathbb{R}^2 : \abs{z_2}<w_m(a)\}
    \end{equation}
    are unstable for $\mathcal{P}_{\alpha}$.
\end{enumerate}
\end{theorem}
Notice that the asymptotic formula \eqref{eq:asymptotic_c_alpha_1} is consistent with Remark \ref{remark:radius_alpha_1}, in which we saw that $R_1(\alpha=1)=0$.
\begin{remark}\label{remark:numerical_alpha_0}
    In fact, as we will see in Section \ref{sec:computations}, the stability of the first branch of near-cylinders, $E_1(a)$, is determined by the sign of an explicit function of $\alpha$. It is given by \eqref{expression_ddot_mu_hom_kernel}, and it is a special function expressed only in terms of the gamma and modified Bessel functions.\footnote{To be more precise, this function depends explicitly on $\alpha$ and $R_1$, and $R_1$ depends implicitly on $\alpha$.} The condition that $\alpha>\alpha_0$ in claim \textit{(2)} of Theorem \ref{theorem_stability_near_cylinders} above comes from the fact that we have rigorously established the behaviour of this function only in the limit $\alpha\uparrow 1$. However, by numerically calculating its values for $\alpha\in (0,1)$ we can see that $E_1(a)$ is in fact unstable for all $\alpha\in (0,1)$ (see Figure~\ref{fig:stability_vs_alpha} in Section~\ref{sec:computations}). This fact could possibly be made rigorous using interval analysis techniques.
\end{remark}

We note that, since our result is perturbative, Theorem \ref{theorem_stability_near_cylinders} states that the near-cylinders $E_m(a)$ are unstable for $\mathcal{P}_{\alpha}$ \textit{for $\abs{a}$ small enough}, i.e., when they are sufficiently ``flat''. For $\abs{a}$ large, it could be the case that they become stable (or even minimizers) for $\mathcal{P}_{\alpha}$. 

Throughout the whole paper, we say that a $2\pi$-periodic, cylindrically symmetric set $E$ with generatrix function $u$ is \textit{$\varepsilon$-$C^{1,\beta}$-flat} if there exists $R>0$ such that $\norm{u-R}_{C^{1,\beta}(\mathbb{T}^1)}<\varepsilon$; in other words, if $E$ is a small $C^{1,\beta}(\mathbb{T}^1)$ perturbation of a straight cylinder.

Now, the uniqueness result from Theorem \ref{theorem_seq_bif} allows us to show that every even, cylindrically symmetric volume-constrained critical point of $\mathcal{P}_{\alpha}$ that is sufficiently $C^{1,\beta}$-flat must be either a straight cylinder or one of the near-cylinders found in Theorem \ref{theorem_seq_bif}. From Theorem \ref{theorem_stability_near_cylinders}, we can then conclude that if such a critical point is also stable, then it must be a straight cylinder (since all non-straight near-cylinders are unstable). We state this result in Corollary \ref{cor::classification_minimizers} below.

We note also that by the results in \cite{CCM1}, we know that every volume-constrained minimizer of $\mathcal{P}_{\alpha}$ must be even and cylindrically symmetric. Hence, we also deduce that every sufficiently $C^{1,\beta}$-flat volume-constrained minimizer of $\mathcal{P}_{\alpha}$ is a straight cylinder. Note, however, that we do \emph{not} make any claims on the minimizing character of straight cylinders of any radius.

The precise statement is the following.
\begin{corollary}\label{cor::classification_minimizers}
    Let $R>0$, let $\alpha_0$ be as in Theorem \ref{theorem_stability_near_cylinders}, and let $\alpha\in [\alpha_0, 1)$ and $\alpha < \beta< \min\{1, 2 \alpha + \frac{1}{2}\}$. There exists $\varepsilon>0$, depending only on $\alpha$, $\beta$, and $R$, such that the following~holds. 

    If $E\subset \mathbb{R}^2$ is an even, cylindrically symmetric, volume-constrained critical point of $\mathcal{P}_{\alpha}$ that is $\varepsilon$-close to the straight cylinder of radius $R$ in $C^{1,\beta}_{even}(\mathbb{T}^1)$, then, either $E$ is a straight cylinder of radius $R'\geq R_1$, or $E$ is unstable, where $R_1$ is given by Theorem \ref{theorem_seq_bif}.

    The same holds for all $\alpha\in (0,1)$ when $R\neq R_1$.
\end{corollary}

We remark that since in Theorem \ref{theorem_stability_near_cylinders} the assumption that $\alpha\geq \alpha_0$ is only needed to determine the stability of the first branch of near-cylinders, $E_1(a)$, Corollary \ref{cor::classification_minimizers} can be stated for all $\alpha\in (0,1)$ if one considers $R\neq R_1$, as we claim in the last part of its statement. Nonetheless, as we explained in Remark \ref{remark:numerical_alpha_0}, the condition that $\alpha\geq \alpha_0$ could probably be removed altogether by using interval analysis techniques to show the instability of the branch $E_1(a)$ for all $\alpha\in (0,1)$.

\definecolor{qqqqff}{rgb}{0,0,1}
\definecolor{ffqqqq}{rgb}{1,0,0}
\definecolor{negre}{rgb}{0,0,0}
\definecolor{taronja}{rgb}{1,0.64,0}
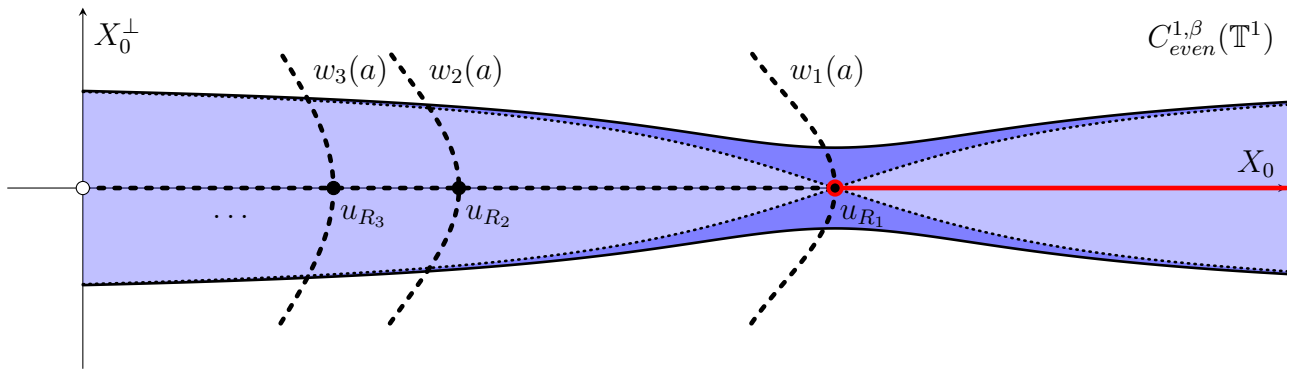
\begin{figure}[h!]
\centering
\begin{tikzpicture}[line cap=round,line join=round,>=triangle 45,x=1cm,y=1cm]
\begin{axis}[
x=1cm,y=1cm,
axis lines=middle,
xmin=-1,
xmax=16,
ymin=-2.4,
ymax=2.4,
xtick={0,10},
xticklabels={},
xlabel={$X_0$},
ytick={-35},
ylabel = {$X_0^\perp$}]
\clip(-1,-2.5) rectangle (16,2.5);

\fill[qqqqff,opacity=0.25,domain=0:10,samples=100] 
  plot(\x,{atan((10-\x)/3)*pi/180}) -- 
  plot[domain=10:0](\x,{-atan((10-\x)/3)*pi/180}) -- cycle;

\fill[qqqqff,opacity=0.25,domain=10:20,samples=100] 
  plot(\x,{atan((10-\x)/3)*pi/180}) -- 
  plot[domain=20:10](\x,{-atan((10-\x)/3)*pi/180}) -- cycle;

\fill[qqqqff,opacity=0.5] 
  plot[domain=0:18,samples=100] (\x,{sqrt((atan((10-\x)/3)*pi/180)^2+0.5)-0.17}) 
  -- plot[domain=18:0,samples=100] (\x,{abs(atan((10-\x)/3)*pi/180)})
  -- cycle;
\fill[qqqqff,opacity=0.5] 
  plot[domain=0:18,samples=100] (\x,{-sqrt((atan((10-\x)/3)*pi/180)^2+0.5)+0.17}) 
  -- plot[domain=18:0,samples=100] (\x,{-abs(atan((10-\x)/3)*pi/180)})
  -- cycle;


\draw[line width=1.75pt,color=black,dash pattern=on 2pt off 4pt,samples=100,domain=-1.8:1.8] plot({10+0.7*(cos(70*\x)-1)}, \x);

\draw[line width=1.75pt,color=black,dash pattern=on 2pt off 4pt,samples=100,domain=-1.8:1.8] plot({5+0.7*(cos(60*\x)-1)}, \x);

\draw[line width=1.75pt,color=black,dash pattern=on 2pt off 4pt,samples=100,domain=-1.8:1.8] plot({3.33+0.7*(cos(50*\x)-1)}, \x);

\draw [samples=500,color=black,dash pattern=on 2pt off 4pt,xshift=0cm,yshift=0cm,line width=1.75pt,domain=0:10)] plot (\x,0);

\draw [samples=500,color=ffqqqq,smooth,xshift=0cm,yshift=0cm,line width=1.75pt,domain=10:16)] plot (\x,0);

\draw[line width=1pt,color=black,dash pattern=on 0.5pt off 2pt,samples=100,domain=0:18] plot(\x,{atan((10-\x)/3)*pi/180});
\draw[line width=1pt,color=black,dash pattern=on 0.5pt off 2pt,samples=100,domain=0:18] plot(\x,{-atan((10-\x)/3)*pi/180});

\draw[line width=1pt,color=black,smooth,samples=100,domain=0:18] plot(\x,{sqrt((atan((10-\x)/3)*pi/180)^2+0.5)-0.17});
\draw[line width=1pt,color=black,smooth,samples=100,domain=0:18] plot(\x,{-sqrt((atan((10-\x)/3)*pi/180)^2+0.5)+0.17});

\begin{scriptsize}
\draw [fill=white] (0,0) circle (2.5pt);

\draw [line width=1.5pt, color=ffqqqq, fill=black] (10,0) circle (2.5pt);

\draw[color=negre] (10.375,-0.375) node {$u_{R_1}$};
\draw [fill=black] (5,0) circle (2.5pt);
\draw[color=black] (4.65+0.75,-0.375) node {$u_{R_2}$};
\draw [fill=black] (3.3333333333333335,0) circle (2.5pt);
\draw[color=black] (3.333+0.4,-0.375) node {$u_{R_3}$};

\draw[color=black] (2,-0.375) node {$\dots$};

\draw[color=negre] (9.9,1.56) node {$w_{1}(a)$};
\draw[color=negre] (5.1,1.56) node {$w_{2}(a)$};
\draw[color=negre] (3.56,1.56) node {$w_{3}(a)$};

\draw[color=negre] (15,2) node {$C^{1,\beta}_{even}(\mathbb{T}^1)$};

\end{scriptsize}
\end{axis}
\end{tikzpicture}

\caption{Representation of the space $C^{1,\beta}_{even}(\mathbb{T}^1)$ decomposed into $X_0$ (the subspace of constant functions) and $X_0^\perp$ (the subspace of functions orthogonal to $1$ in $L^2_{even}(\mathbb{T}^1)$). The horizontal axis corresponds to straight cylinders (that is, the trivial solutions $u_R$), and the vertical curves at each $u_{R_m}$ correspond to the near-cylinders of Theorem \ref{theorem_seq_bif}. As stated in Theorem \ref{theorem_stability_near_cylinders}, straight cylinders lying to the right of $u_{R_1}$ are stable, whereas those lying to its left are unstable. The dashed bifurcated curves are all unstable.
Corollary~\ref{cor::classification_minimizers} states that all critical points different from straight cylinders lying in the two shaded areas in the Figure are unstable when $\alpha\geq \alpha_0$. When $\alpha<\alpha_0$, we expect the same result to hold, but we can only rigorously conclude their instability in the lighter shaded area. Notice the pitchfork-like bifurcation occurring at $u_{R_1}$.}
\label{fig1}
\end{figure}

Prior to our work, Cabré, Csató, and Mas \cite{CabreCsatoMas_stability} showed that in every dimension $n\geq 2$ and for every $\alpha \in (0,1)$ there exists a critical radius $R_{\alpha}>0$ such that, for every $R>0$, the straight cylinder of radius $R$ is a stable constant-NMC surface of revolution if, and only if, $R\geq R^*$. Moreover, they proved that $R_{\alpha}^2\to n-2$ as $\alpha \uparrow 1$. Expression \eqref{eq:asymptotic_R1_alpha_1} in the first part of Theorem \ref{theorem_stability_near_cylinders} makes the rate of convergence precise in dimension $n=2$. Regarding the shape of minimizers, Dávila, Del Pino, Dipierro, and Valdinoci \cite{DavilaDelPinoDipierroValdinoci2016} proved, in every dimension $n\geq 2$ and for every $\alpha \in (0,1)$, that all volume-constrained minimizers of $\mathcal{P}_{\alpha}$ of sufficiently small volume must be close to an array of spheres in a measure-theoretic sense. In particular, their result shows that straight cylinders are not minimizers for small volumes. Corollary \ref{cor::classification_minimizers} above gives a complementary picture in dimension $n=2$ in the range of large volumes. Hence, our results support the conjecture that, like in the local case, volume-constrained minimizers of $\mathcal{P}_{\alpha}$ are either arrays of spheres or straight cylinders, depending on the size of the volume constraint.

This paper is structured as follows. In Section \ref{section:a_seq_of_bif_delaunay} we introduce the functional framework in which we will work, and we prove Theorem \ref{theorem_seq_bif}. In Section \ref{section:on_the_stab_of_periodic_NMC_surf_of_rev} we define the notion of stability for constant-NMC surfaces of revolution, and we give a simple characterization in terms of a quadratic form. Next, in Section \ref{section:on_the_eigenvalues}, we derive some important properties of the eigenvalues of the linearized NMC operator on the near-cylinders, which we then use in Section \ref{sec:stability_of_near_cyls} to prove Theorem \ref{theorem_stability_near_cylinders} and Corollary~\ref{cor::classification_minimizers}. Finally, in Section \ref{sec:computations} we include some the lengthy computations needed to prove the key results in sections \ref{section:on_the_eigenvalues} and \ref{sec:stability_of_near_cyls}.

\section{A sequence of bifurcated families of Delaunay near-cylinders in $\mathbb{R}^2$}\label{section:a_seq_of_bif_delaunay}
This section is devoted to the proof of Theorem \ref{theorem_seq_bif}. Moreover, here we introduce some results and notation that will be used throughout the whole paper.

Let $E = \{ (z_1,z_2)\in \mathbb{R}^2 : \abs{z_2}<u(z_1)\}$ for some positive function $u$, not necessarily periodic. Like in the introduction, we denote, with a small abuse of notation,
\begin{equation}\label{def:NMC_function}
    H_\alpha(u)(s):=H_\alpha[E](s,u(s)).
\end{equation}
The following lemma gives a simpler expression for the nonlocal mean curvature of $E$. We refer the reader to \cite[Lemma 4.1]{CNLMCDelCyl} or \cite[Lemma 3.1]{AlcoverBruera} for its simple proof.
\begin{lemma}[{\cite[Lemma 4.1]{CNLMCDelCyl}}] \label{NLMCLemma}
    Let $K(z)=\abs{z}^{-2-\alpha}$, with $\alpha \in (0,1)$. The nonlocal mean curvature of $E$ at the point $(s,u(s))\in \partial E$ is given by
    \begin{equation}\label{eq:simple_expr_NMC}
        \frac{1}{2} H_{\alpha}(u)(s) = \int_\mathbb{R} G\left(t,u(s)-u(s-t)\right) \dd{t}
        - \int_\mathbb{R} \left\{ G\left(t,u(s)+u(s-t)\right) - G\left(t,+\infty\right) \right\} \dd{t},
    \end{equation}
    where the integrals are to be understood in the principal value sense, and
    \begin{equation}
        G(\tilde{t},v) := \int_0^v K(\tilde{t},\tau) \dd{\tau}.
    \end{equation}
\end{lemma}
With this result in hand, it is immediate to check that the NMC of a straight cylinder is constant and positive. Indeed, for a straight cylinder of radius $R$ we have $u=u_R\equiv R$, and thus
\begin{equation}\label{eq:expression_hR}
    h_R := H_{\alpha}(u_R) = -2 \int_\mathbb{R} \left\{ G\left(t,2R\right) - G\left(t,+\infty\right) \right\} \dd{t},
\end{equation}
which is positive because the function $G(t,\cdot)$ is increasing.

Now, to apply the implicit function theorem to the NMC operator, which is the key idea in the proof of Theorem \ref{theorem_seq_bif}, and to study the stability of the near-cylinders in Section \ref{sec:stability_of_near_cyls}, we need the NMC operator to be sufficiently smooth when viewed as an operator acting on functions in the sense of \eqref{def:NMC_function}. Thankfully for us, the regularity of the NMC operator was already proved by Cabré, Fall, and Weth in \cite[Section 5]{CabreFallWeth2018}. We state the precise result next.
\begin{lemma}[{\cite[Section 5]{CabreFallWeth2018}}]\label{lemma_H_Cinfty}
    Let $\alpha\in (0,1)$ and $\beta\in (\alpha, 1)$. The map
    \begin{equation}
        H_{\alpha} : \mathcal{O}\subset C^{1,\beta}(\mathbb{R}) \to C^{\beta-\alpha}(\mathbb{R}),
    \end{equation}
    with $\mathcal{O} = \{u\in C^{1,\beta}(\mathbb{R}) : \inf_{\mathbb{R}} u > 0\}$ and $H_{\alpha}$ as defined in \eqref{def:NMC_function}, is of class $C^\infty$.
\end{lemma}

The following key lemma states that the linearization of the NMC operator on every straight cylinder diagonalizes in the Fourier basis of cosines, and it gives some important properties of its eigenvalues. This result was already proved in \cite[Proposition 5.1 and Lemma 5.3]{CNLMCDelCyl}. We include its proof here, since we will use it extensively throughout the paper.
\begin{lemma}[{\cite[Proposition 5.1 and Lemma 5.3]{CNLMCDelCyl}}]\label{lemma:diagonalization_L}
    Let $K(z)=\abs{z}^{-2-\alpha}$, with $\alpha\in (0,1)$, and let $R>0$. The functions $e_k:=\cos(k\cdot)$ are eigenfunctions of the operator
    \begin{equation}
    L_R\psi:=\frac{1}{2}D_u H_{\alpha}(R)\psi = \int_{\mathbb{R}} (\psi(s)-\psi(s-t))K(t,0) \dd{t} - \int_{\mathbb{R}} (\psi(s)+\psi(s-t))K(t,2R)\dd{t}
    \end{equation}
    with eigenvalues
    \begin{equation}\label{eq::expression_eigenvalues}
        \mu_k(R) := \int_{\mathbb{R}}(1-\cos(kt))K(t,0)\dd{t} - \int_{\mathbb{R}}(1+\cos(kt))K(t,2R)\dd{t}.
    \end{equation}

    Moreover, the following hold:
    \begin{enumerate}
        \item For every $R>0$, the eigenvalues form an increasing sequence, i.e., $\mu_0 < \mu_1 < \mu_2 < \cdots$, and $\mu_0<0$.
        \item For every $R>0$, the limit $\lim_{k\to +\infty} k^{-1-\alpha}\mu_k(R)$ is positive and finite.
        \item For every $k\geq 0$, $\mu_k(R)$ is an increasing function of $R$. As a consequence, for every $m\geq 1$ there exists a unique $R_m>0$ such that $\mu_m(R_m)=0$. Moreover, $R_m = R_1/m$.
    \end{enumerate}
\end{lemma}
\begin{proof}
    The first part of the lemma is an easy consequence of Lemma \ref{NLMCLemma} and of the identities
    \begin{equation}
        \cos(s)\pm\cos(s-t) = (1\pm\cos(t))\cos(s) \pm \sin(s)\sin(t).
    \end{equation}
    Indeed, differentiating \eqref{eq:simple_expr_NMC} at $u\equiv R$ and applying the previous identity we obtain
    \begin{align}
        L_R[\cos(k\cdot)](s) &= \left(\int_{\mathbb{R}} (1-\cos(kt))K(t,0)\dd{t} - \int_{\mathbb{R}}(1+\cos(kt))K(t,2R)\dd{t}\right)\cos(ks) \\[4 pt] &= \mu_k(R)\cos(ks),
    \end{align}
    since the integrals of $\pm \sin(s)\sin(t)$ against $K(t,0)$ and $K(t,2R)$, respectively, vanish (in the principal value sense) because the former are odd functions of $t$ and the latter are even in $t$.

    To prove \textit{(1)}, notice first that for $k=0$ we have
    \begin{equation}
        \mu_0(R) = -\int_{\mathbb{R}}2K(t,2R)\dd{t} < 0.
    \end{equation}

    To see that $(\mu_k(R))$ is an increasing sequence, notice that, changing variables in the expression for $\mu_k$ and using the homogeneity of $K$ we obtain
    \begin{equation}\label{eq:mu_k_change_vars}
        \mu_k(R) = k^{1+\alpha}\left(\int_{\mathbb{R}} (1-\cos(\bar{t}))K(\bar{t},0)\dd{\bar{t}} - \int_{\mathbb{R}}(1+\cos(\bar{t}))K(\bar{t},2Rk)\dd{\bar{t}}\right).
    \end{equation}
    Now, since $K_{z_2}(z)=-(2+\alpha)\abs{z}^{-4-\alpha}z_2<0$ for all $z\in \mathbb{R}^2$ with $z_2>0$, the sequence
    \begin{equation}
        \eta_k := \int_{\mathbb{R}} (1-\cos(\bar{t}))K(\bar{t},0)\dd{\bar{t}} - \int_{\mathbb{R}}(1+\cos(\bar{t}))K(\bar{t},2Rk)\dd{\bar{t}}
    \end{equation}
    is increasing in $k$ for every $R>0$. Hence, we conclude that the sequence $(\mu_k)$ is also increasing. 

    Letting $k\to +\infty$ in \eqref{eq:mu_k_change_vars} and applying the dominated convergence theorem yields \textit{(2)}.

    Now, to prove \textit{(3)}, notice that, again, since $K_{z_2}(z)<0$ whenever $z_2>0$, we have that every eigenvalue $\mu_k(R)$ is an increasing function of $R$. On the other hand, we have that
    \begin{equation}
        \lim_{R\downarrow 0} \mu_k(R) = -\infty \quad \text{and}\quad \lim_{R\uparrow +\infty} \mu_k(R) = \int_{\mathbb{R}}(1-\cos(kt))K(t,0)\dd{t} >0.
    \end{equation}
    Therefore, by Bolzano's theorem, we conclude that the equation $\mu_k(R)=0$ has a unique solution for every $k\geq 1$.

    To see that that $R_m = R_1/m$ for every $m\geq 1$, we change variables in the equation $\mu_m(R_m)=~0$ and use the homogeneity of $K$ to obtain
    \begin{equation}\label{eq_for_Rn_after_change_var}
    \int_{\mathbb{R}} \left\{(1-\cos(\bar{t}))K\left(\bar{t},0\right) - (1+\cos(\bar{t}))K\left(\bar{t},2mR_{m}\right) \right\}\dd{\bar{t}}=0,
    \end{equation}
    which is precisely the equation for $R_1$. In other words, $mR_m$ solves $\mu_1(mR_m)=0$. Since the equation $\mu_1(R)=0$ has a unique solution, we conclude that $mR_m = R_1$.
\end{proof}

We now have all the necessary ingredients to prove Theorem \ref{theorem_seq_bif}.
\begin{proof}[Proof of Theorem \ref{theorem_seq_bif}]
We divide the proof into several steps.

\proofstep{Step 1. Existence of bifurcated curves of solutions.} The proof of existence of the bifurcation points and the construction of the bifurcated families is essentially the same as the one of \cite[Theorem 1.2]{CNLMCDelCyl}---see also \cite[Theorem 1.3]{AlcoverBruera} for a presentation more similar to the one in this paper---, where the existence of the first branch ($m=1$) is addressed. We sketch here the main ideas.

The proof is an application of the implicit function theorem to the equation 
\begin{equation}
    \Phi(a, \gamma, \varphi) = 0,
\end{equation}
with
\begin{equation}
    \Phi : \mathbb{R}\times \mathbb{R}\times C^{1,\beta}_{even}(\mathbb{T}^1)\to C^{\beta - \alpha}_{even}(\mathbb{T}^1)
\end{equation}
defined by
\begin{align}
    \Phi(a,\gamma,\varphi) &:= \frac{1}{2a}\left\{H_{\alpha}(\gamma R + a \varphi) - H_{\alpha}(\gamma R)\right\}\\
    &= \begin{multlined}[t][0.8\displaywidth]
    \int_\mathbb{R} \frac{1}{a} G\left(t,a (\varphi(s)-\varphi(s-t))\right) \dd{t}\\ 
    - \int_{\mathbb{R}} \frac{1}{a} \left\{ G\left(t,2\gamma R + a(\varphi(s)+\varphi(s-t))\right) - G\left(t,2\gamma R\right) \right\}\dd{t}\end{multlined}
\end{align}
whenever $a\neq 0$, and by
\begin{equation}
    \Phi(0,\gamma, \varphi):=\frac{1}{2}D_{u}H_{\alpha}(\gamma R)\varphi = \int_{\mathbb{R}} (\varphi(s)-\varphi(s-t))K(t,0)-\int_{\mathbb{R}}(\varphi(s)+\varphi(s-t))K(t,2 \gamma R)\dd{t}
\end{equation}
when $a=0$, for some appropriately chosen value of $R$, specified next. It is easy to see by changing variables in expression \eqref{eq:simple_expr_NMC} that the operator $H_{\alpha}$ (and thus also $\Phi$) maps $2\pi$-periodic, even functions to $2\pi$-periodic, even functions. Hence, by Lemma \ref{lemma_H_Cinfty}, we will see later that the operator $\Phi$ is well defined and smooth in an appropriate open set.

For every $m\geq 1$, we set $R=R_m=R_1/m$, as given by Lemma \ref{lemma:diagonalization_L} above, and we look for solutions of the equation $\Phi(a,\gamma, \varphi)=0$ that are of the form
\begin{equation}
    \varphi(s) = \cos(ms) + v(s)
\end{equation}
for some $v\in C^{1,\beta}_{even}(\mathbb{T}^1)$ orthogonal to $\cos(m\cdot)$ in $L^2_{even}(\mathbb{T}^1)$. In the sequel, we denote 
\begin{equation}
\overline{\Phi}_m(a, \gamma, v) := \Phi(a, \gamma, \cos(m\cdot)+v),
\end{equation}
with $R=R_m$. Notice that for this value of $R$, by Lemma \ref{lemma:diagonalization_L} we have that
\begin{equation}
    \overline{\Phi}_m(0,1,0)=\frac{1}{2}D_u H_{\alpha}(R_m)\cos(m\cdot) = 0.
\end{equation}
Moreover, by Lemma \ref{lemma_H_Cinfty}, the operator $\overline{\Phi}_m$ is $C^\infty$ in a neighbourhood of $(0,1,0)$.

To apply the implicit function theorem, it only remains to check that the linearized operator $D_{(\gamma,v)}\overline{\Phi}_m$ is invertible at $(0,1,0)$. We have
\begin{equation}
  D_{v}\overline{\Phi}_m(0,1,0) \psi(s) = L_R\psi(s)=\int_{\mathbb{R}} \{(\psi(s)-\psi(s-t))K(t,0)- (\psi(s)+\psi(s-t))K(t,2R_m)\} \dd{t}
\end{equation}
and
\begin{equation}\label{eq:der_phi_gamma}
    D_{\gamma}\overline{\Phi}_m(0,1,0) (s)= -\int_{\mathbb{R}}(\cos(ms)+\cos(m(s-t)))2R_mK_{z_2}(t,2R_m) \dd{t}=: \kappa_m \cos(m s),
\end{equation}
with
\begin{equation}
    \kappa_m = -\int_{\mathbb{R}} (1+\cos(mt))2R_mK_{z_2}(t,2R_m) \dd{t}>0
\end{equation}
because $K_{z_2}(z)=-(2+\alpha)\abs{z}^{-4-\alpha}z_2<0$ for all $z\in \mathbb{R}^2$ with $z_2>0$.

By Lemma \ref{lemma:diagonalization_L}, the operator $L_R$ diagonalizes in $H^{1+\alpha}_{even}(\mathbb{T}^1)$ (the fractional Sobolev space of $2\pi$-periodic, even functions) in the basis of eigenfunctions $\cos(k \cdot)$ with eigenvalues $\mu_k$ given by \eqref{eq::expression_eigenvalues} with $R=R_m$, and, moreover, $\mu_k(R)\neq 0$ for all $k\neq m$. As a consequence (see Step 4 in the proof of \cite[Proposition 3.3]{AlcoverBruera} for a proof of the following facts), the operator $L$ can be shown to be invertible when considered as an operator between the spaces $V_1^\perp\cap H^{1+\alpha}_{even}(\mathbb{T}^1)$ and $V_1^\perp$, with $V_1^\perp := \{v \in L^2_{even}(\mathbb{T}^1) : \expval{v,\cos(m\cdot)}=0\}$. Using regularity estimates for the fractional Laplacian, one can see that it is in fact also invertible when considered as an operator between the spaces $V_1^{\perp}\cap C^{1,\beta}_{even}(\mathbb{T}^1)$ and $V_1^\perp\cap C^{\beta-\alpha}_{even}(\mathbb{T}^1)$. Since, moreover, we have that $D_{\gamma}\overline{\Phi}_m(0,1,0)\in \mathrm{span}\{\cos(m\cdot)\}$, it follows that the full differential $D_{(\gamma, v)}\overline{\Phi}_m$ is invertible as an operator between $C^{1,\beta}_{even}(\mathbb{T}^1)$ and $C^{\beta- \alpha}_{even}(\mathbb{T}^1)$.

Hence, the existence of the smooth curve of solutions $\{(\gamma_m(a)R_m, w_m(a)):a\in (- \nu_m, \nu_m)\}$ follows from the implicit function theorem applied to the operator $\overline{\Phi}_m$ at $(0,1,0)$. Note that, by construction, for every $m\geq 1$ and every $a\in (-\nu_m, \nu_m)$, the cylindrically symmetric set generated by $w_m(a)$ has constant NMC equal to $h_{\gamma_m(a)R_m}=H_{\alpha}(\gamma_m(a) R_m)$.

\proofstep{Step 2. Local uniqueness of solutions.}
Since, by Lemma \ref{lemma:diagonalization_L}, the eigenvalues $\mu_k(R)$ are increasing in $R$ for every $k \geq 0$, and in $k$ for every $R > 0$, by the same arguments as in the end of Step 1 we deduce, on the one hand, that the operator $D_u H_{\alpha}(u_R)$ is invertible for every $R\in (R_m,R_{m-1})$ and every $m\geq 1$ (with the understanding that $R_{0}=+\infty$). On the other hand, when $R=R_m$, by Step 1 we know that the operator $D_{(\gamma, \varphi)}\overline{\Phi}_m(0,1,0)$ is invertible for every $m\geq 1$. The claim on uniqueness is therefore a direct consequence of the implicit function theorem applied to the operators $H_{\alpha}$ and $\overline{\Phi}_m$ at $u\equiv R \in (R_m, R_{m-1})$ and $(0,1,0)$, respectively, for every $m\geq 1$.
\end{proof}

The following proposition gathers a few properties of the curves of bifurcated solutions. Item \textit{(2)} in particular---more precisely, the fact that $\derivative{a}\big\vert_{a=0}\gamma_m = 0$---will be an important ingredient in our proof of Proposition \ref{proposition:stab_straight_cylinders}, where it will allow us to simplify some calculations.
\begin{proposition}\label{proposition:dot_gamma}
In the setting and the notation of Theorem \ref{theorem_seq_bif}, for every $m\geq 1$ and every $a\in (-\nu_m, \nu_m)$, writing $m=2^{p}q$ for some integers $p\geq 0$ and $q$ odd,
\begin{enumerate}
    \item The function $w_m(a)$ has prime period equal to $2\pi/(2^p\ell)$ for some positive integer $\ell$ dividing $q$. As a consequence, the Fourier series of $v_m(a)$ is given by
    \begin{equation}
        v_m(a)(s) = \sum_{k\neq q/\ell} c_{2^p \ell k}(a) \cos(2^p \ell k s).
    \end{equation}
    \item The following relations hold:
    \begin{equation}
        \gamma_m(a)=\gamma_m(-a) \quad \text{and} \quad v_m(a)(\cdot+\pi/2^p)=- v_m(-a)(\cdot).
    \end{equation}
    In particular, $\derivative{a}\big\vert_{a=0}\gamma_m=0$.
\end{enumerate}
\end{proposition}
\begin{proof}
We prove \textit{(2)} first. For every $m\geq 1$ and every $a\in (-\nu_m, \nu_m)$, since $w_m(a)$ is $2\pi$-periodic and even, $w_m(a)(\cdot + \pi)$ is also $2\pi$-periodic, even, and solves $\Phi(a,\gamma_m(a),w_m(a)(\cdot + \pi))=0$ for $R=R_m$. Moreover,
\begin{align}
    w_m(a)(s+\pi) &= \gamma_m(a)R_m + a (\cos(m(s+\pi))+v_m(a)(s+\pi))\\ & = \gamma_m(a)R_m + a((-1)^{m}\cos(mx)+v_m(a)(s+\pi)).
\end{align}
Hence, by the local uniqueness of solutions, we deduce that
\begin{align}
    w_m(a)(s+\pi)&=\gamma_m(a)R_m + a((-1)^{m}\cos(ms)+v_m(a)(s+\pi)) \\
    &=\gamma_m(a')R_m + a'(\cos(ms)+v_m(a')(s)) = w_m(a')(s)
\end{align}
for some $a'\in (-\nu_m,\nu_m)$. By orthogonality, we conclude that $a'=(-1)^m a$. We must now consider two cases.

If $m$ is odd, we deduce that $a'=-a$. Therefore, we find that $\gamma_m(a)=\gamma_m(- a)$ and that $v_m(a)(\cdot+\pi)=- v_m(- a)(\cdot)$, which completes the proof of \textit{(2)}.
    
If $m$ is even, that is, if $m=2m'$ for some integer $m'\geq 1$, we find that $a'=a$ and, therefore, that $w_m(a)(\cdot +\pi) = w_m(a)(\cdot)$; in other words, $w_m(a)$ is $\pi$-periodic. In this case, we repeat the argument with $w_m(a)(\cdot+\pi/2)$. 

Since $w_m(a)$ is $\pi$-periodic, even, and a solution of $\Phi(a, \gamma_m(a), w_m(a))=0$ for $R=R_m$, then $w_m(a)(\cdot + \pi/2)$ is so too and thus, by uniqueness,
\begin{align}
    w_m(a)\left(s+ \frac{\pi}{2}\right) &= \gamma_m(a)R_m + a\left((-1)^{m'} \cos(m s)+v_m(a)\left(s+\frac{\pi}{2}\right)\right) \\ &= \gamma_m(a')R_m + a'(\cos(ms)+v_m(a')(s)) = w_m(a')(s).
\end{align}
Again, by orthogonality, we see that $a'=(-1)^{m'} a$.

As before, if $m'=m/2$ is odd, we deduce that $a'=-a$. Therefore, we find that $\gamma_m(a)=\gamma_m(-a)$ and that $v_m(a)(\cdot + \pi/2)=-v_m(-a)(\cdot)$, which completes the proof of \textit{(2)}. If, on the contrary, $m'=m/2$ is even, we find that $w_m(a)(\cdot + \pi/2)=w_m(a)(\cdot)$.

The claim for arbitrary $m=2^p q$ follows easily by induction on $p$. Notice that from this we deduce, in particular, that $w_m(a)$ is $2\pi/2^p$-periodic.

To show \textit{(1)}, let $T$ be the prime period of $w_m(a)$. We must have
\begin{align}
    \cos(m s) + v_m(a)(s) &= \cos(m (s+T)) + v_m(a)(s+T) \\ 
    &= \begin{multlined}[0.65\displaywidth][t]
    \cos(ms)\cos(mT)-\sin(ms)\sin(mT) \\ + \sum_{k\neq m} c_k(a) \left\{\cos(ks)\cos(kT)-\sin(ks)\sin(kT)\right\}.
    \end{multlined}
\end{align}
By orthogonality with $\cos(m\cdot)$, we deduce that $\cos(mT)=1$, and thus $T=2\pi r/m=2\pi r/(2^p q)$ for some integer $r\geq 1$. Since the prime period must divide $2\pi/2^p$, we find that $(2\pi /2^p)/T=q/r$ must be an integer, say, $\ell$. Similarly, by orthogonality with $\cos(k\cdot)$, for every $k\geq 1$ we deduce that, if $c_k(a)\neq 0$, then $\cos(kT)=1$, and thus $k=(2^p \ell) k'$ for some integer $k'\geq 1$.
\end{proof}

\section{On the stability of periodic constant-NMC surfaces of revolution}\label{section:on_the_stab_of_periodic_NMC_surf_of_rev}
In \cite{CCM1}, the authors proved that volume-constrained minimizers of the periodic nonlocal perimeter $\mathcal{P}_{\alpha}$ have constant NMC. In other words, the constant-NMC equation arises as the Euler-Lagrange equation for the periodic nonlocal perimeter functional, $\mathcal{P}_{\alpha}$. It is therefore natural to ask whether a given constant-NMC surface is a volume-constrained minimizer of $\mathcal{P}_{\alpha}$, something of great difficulty in general. An easier task is to study their stability (note that unstable configurations cannot be minimizers). Most of the ideas presented in this section are due to Cabré, Csató, and Mas, who introduced them in \cite{CabreCsatoMas_stability} for their study of the stability of straight cylinders as constant-NMC surfaces of revolution in $\mathbb{R}^n$.

Even though the goal of this paper is to study the stability of near-cylinders in dimension $n=2$, since the ideas and definitions in this section are the same in all dimensions, we state them for all $n\geq 2$.

To define stability of a constant-NMC set, we need to introduce the class variations that we will consider admissible.
\begin{definition}
Let $E\subset \mathbb{R}^n$ be an open set that is $2\pi$-periodic in the $x_1$-direction. We say that a collection of open sets $\{E_t\}_{t\in (- \varepsilon, \varepsilon)}$ is a \textit{volume-preserving periodic variation} of $E$ if there exists a family of diffeomorphisms $\Phi_t : \mathbb{R}^n \to \mathbb{R}^n$ parametrized by $t\in (- \varepsilon, \varepsilon)$ such that $E_t = \Phi_t(E)$, and the following properties hold:
\begin{itemize}
    \item $\Phi_0 = \mathrm{id}_{\mathbb{R}^n}$.
    \item $\Phi \in C^2\left((- \varepsilon, \varepsilon)\times \mathbb{R}^n\right)$,
    \item $\Phi_t$ is volume preserving: for every $t\in (-a,a)$, one has that $\abs{E_t \cap \Omega} = \abs{E\cap \Omega}$, where $\abs{\cdot}$ is the Lebesgue measure, and $\Omega = \{(x_1,x')\in \mathbb{R}\times \mathbb{R}^{n-1} : -\pi < x_1 < \pi\}$,
    \item $\Phi_t$ preserves the periodicity of $E$: for every $t\in (-\varepsilon,\varepsilon)$ and every $(x_1,x')\in \mathbb{R}\times \mathbb{R}^{n-1}$, one has that $\Phi_t( x_1+2\pi, x') = \Phi_t(x_1,x')+2\pi e_1$.
\end{itemize}
If, in addition, $\Phi_t$ is of the form $\Phi_t(x_1,x') = (x_1, \eta_t(x_1)x')$ for some function $\eta \in C^2\left((-\varepsilon,\varepsilon)\times \mathbb{R}\right)$ such that $\eta_t(\cdot)$ is $2\pi$-periodic and even for every $t\in (-\varepsilon,\varepsilon)$, and $\eta_0(x_1)=1$, we say that the variation is \textit{cylindrically symmetric} and \textit{even}.
\end{definition}

In \cite{CCM1}, the authors proved that every volume-constrained minimizer of $\mathcal{P}_{\alpha}$ must be cylindrically symmetric and even---i.e., a Delaunay surface. Therefore, in the definition of stability for even, cylindrically symmetric constant-NMC sets, which is given next, we shall only consider even, cylindrically symmetric variations. 
\begin{definition}
Let $\alpha\in (0,1)$, and let $E\subset \mathbb{R}^n$ be an open, $2\pi$-periodic, even, cylindrically symmetric open set. We say that $E$ is a volume-constrained critical point of $\mathcal{P}_\alpha$ if
\begin{equation}
    \derivative{}{t} \mathcal{P}_\alpha[E_t]\Big\vert_{t=0} = 0
\end{equation}
for all even, cylindrically symmetric volume-preserving periodic variations $E_t$ of $E$. We say that $E$ is \textit{stable} for $\mathcal{P}_{\alpha}$ if, moreover,
\begin{equation}
    \derivative[2]{}{t} \mathcal{P}_\alpha[E_t]\Big\vert_{t=0} \geq 0
\end{equation}
for all even, cylindrically symmetric volume-preserving periodic variations $E_t$ of $E$.\footnote{Often, in the literature, critical points of a given functional are called \textit{stable} if the second variation of the functional at said point is strictly positive, and they are called \textit{semi-stable} if it is nonnegative.}
\end{definition}
In the following, whenever there is no ambiguity, we will refer to volume-constrained critical points of $\mathcal{P}_\alpha$ simply as ``critical points'', and to even, cylindrically symmetric volume-preserving periodic variations simply as ``variations'' to abbreviate.

\begin{remark}\label{remark_constant NMC_critical_point}
Every $2\pi$-periodic open set $E\subset \mathbb{R}^n$ such that $H_\alpha[E]$ is constant on $\partial E$ is a volume-constrained critical point of $\mathcal{P}_\alpha$, and vice versa (provided that $E$ is appropriately smooth). When $E$ is cylindrically symmetric, this follows from \eqref{eq:first_var_Palpha} and \eqref{volume_preserving_condition_vt} below for cylindrically symmetric variations. However, the first variation of $\mathcal{P}_{\alpha}$ at $E$ vanishes for \textit{all} volume-preserving periodic variations (not necessarily even or cylindrically symmetric). This fact follows from the formula
\begin{equation}
    \derivative{t}\mathcal{P}_{\alpha}[E_t] = \int_{\partial E_t \cap \Omega} \xi_t(x) H_{\alpha}[E_t](x)  \dd{\mathcal{H}^{n-1}}(x), \quad \text{with }\xi_t:= \derivative{\Phi_t}{t}\cdot \nu_{E_t},
\end{equation}
which can be obtained by adapting the first variation formula in \cite[Theorem 6.1]{FigalliFuscoEtAl_2015}; see \cite[Section 2]{CabreCsatoMas_stability} for more details.
\end{remark}

\subsection{A Rayleigh quotient criterion for stability}
In this subsection we show that the stability of Delaunay surfaces is determined by the quadratic form associated to the linearization of the NMC operator.

Throughout the rest of the section, we denote by $E\subset \mathbb{R}^n$ an even, cylindrically symmetric measurable set of the form
\begin{equation}
    E=\{(x_1,x')\in \mathbb{R}^n : \abs{x'}<u(x_1)\},
\end{equation}
with $u:\mathbb{R}\to \mathbb{R}$ $2\pi$-periodic, even, positive, and such that $E$ has constant NMC. We consider even, cylindrically symmetric variations of $E$, namely, $E_t=\{(x_1,x')\in \mathbb{R}^n : \abs{x'}<u_t(x_1)\}$. From \cite[Part 3 in the proof of Theorem 1.1]{CCM1} we have that
\begin{equation}\label{eq:first_var_Palpha}
    \derivative{}{t}\mathcal{P}_\alpha[E_t] = \abs{\mathbb{S}^{n-2}}\int_{-\pi}^{\pi} v_t(x_1)u_t(x_1)^{n-2}H_\alpha(u_t)\dd{x_1}
\end{equation}
and, denoting $V(t)$ the volume of $E_t$ in $\Omega$, since the variation is volume preserving,
\begin{equation}\label{volume_preserving_condition_vt}
    0=\derivative{}{t}V(t) = (n-1)\abs{B_1^{n-1}}\int_{-\pi}^\pi  v_t(x_1)u_t^{n-2}(x_1)\dd{x_1},
\end{equation}
with $v_t := \derivative{}{t}u_t$. Therefore,
\begin{align}
      \derivative[2]{}{t}\Big\vert_{t=0} \mathcal{P}_\alpha[E_t] &= \derivative{}{t}\Big\vert_{t=0}\abs{\mathbb{S}^{n-2}}\int_{-\pi}^{\pi}  v_t(x_1)u_t(x_1)^{n-2}H_\alpha(u_t)\dd{x_1} \\ 
      &= 
      \begin{multlined}[t][0.7\displaywidth]
      \abs{\mathbb{S}^{n-2}}\int_{-\pi}^{\pi}v_0(x_1)u(x_1)^{n-2} \derivative{}{t}\Big\vert_{t=0} H_\alpha(u_t) \dd{x_1}  \\ 
      + \abs{\mathbb{S}^{n-2}}H_\alpha(u)\int_{-\pi}^{\pi}  \derivative{}{t}\Big\vert_{t=0} \left(v_t(x_1)u_t(x_1)^{n-2}\right)\dd{x_1}.
      \end{multlined}
\end{align}
The second term in the last expression is precisely $V''(0)$, which, since the variation is volume-preserving, is equal to $0$. Hence, from this and \eqref{volume_preserving_condition_vt}, we deduce that
\begin{equation}\label{eq:p''(0)}
    \derivative[2]{}{t}\Big\vert_{t=0} \mathcal{P}_\alpha[E_t] = \int_{-\pi}^{\pi} v_0(x_1)\left(D_u H_\alpha(u)v_0(x_1)\right) \dd{S_u (x_1)}
\end{equation}
and
\begin{equation}\label{eq:vol_pres}
    \int_{-\pi}^\pi  v_0(x_1)\dd{S_u}(x_1) = 0,
\end{equation}
where
\begin{equation}
    \dd{S_u(x_1)} := \abs{\mathbb{S}^{n-2}} u(x_1)^{n-2} \dd{x_1}
\end{equation}
is the surface measure on $\partial E$. Notice that when $n=2$, or when $u$ is constant, $S_u$ is proportional to the $(n-1)$-dimensional Hausdorff measure.

This derivation leads to the following lemma, which was first proved in \cite[Lemma 2.1]{CabreCsatoMas_stability} in the case that $E$ is a straight cylinder. It gives a characterization of stability for Delaunay surfaces in terms of a ``Rayleigh quotient'' associated to the linearization of the NMC operator. 

We note that the quantity $\mathcal{R}$ in \eqref{def:rayleigh_quot} in general is \emph{not} an eigenvalue of $L_u$, even in dimension $n=2$. As we will see in Section \ref{sec:stability_of_near_cyls}, this is because the subspace of zero-mean functions (with respect to the measure $\dd{S_u}$) may not be invariant under $L_u$. Instead, it is an eigenvalue of the composition of $L_u$ with the projection onto this subspace of zero-mean functions---this composition is usually called the \textit{compression} of the operator $L_u$.

\begin{lemma}\label{lemma_cylindrical}
Let $\alpha \in (0,1)$, and let
\begin{equation}
  E = \{(x_1,x')\in \mathbb{R}^{n+1} : \abs{x'}<u(z_1)\},
\end{equation}
with $u: \mathbb{R}\to \mathbb{R}$ a $2\pi$-periodic, even, positive function of class $C^{1, \beta}$ for some $\beta> \alpha$, be such that $E$ has constant nonlocal mean curvature. 

Then, $E$ is stable for $\mathcal{P}_{\alpha}$ with respect to even, cylindrically symmetric volume-preserving periodic variations if, and only if, 
\begin{equation}\label{def:rayleigh_quot}
     \mathcal{R}_u:=\inf_{v}\frac{\int_{-\pi}^\pi vL_uv \dd{S_u}}{\int_{-\pi}^\pi v^2 \dd{S_u}}\geq 0,
\end{equation}
where the infimum is taken over all functions $v\in C^1_{even}(\mathbb{T}^1)$ such that
\begin{equation}
    \int_{-\pi}^\pi v\dd{S}_u=0,
\end{equation}
 $L_u = D_u H_{\alpha}(u)$ is the linearization of the nonlocal mean curvature operator at $E$, and 
\begin{equation}
    \dd{S_u}(x_1) = \abs{ \mathbb{S}^{n-2}}u(x_1)^{n-2}\dd{x_1}
\end{equation}
is the surface measure on $\partial E$.
\end{lemma}
\begin{proof}
    The only thing left to prove is that for every $v\in C^1_{even}(\mathbb{T}^1)$ there exists an even, cylindrically symmetric volume-preserving periodic variation $u_t$ of $E$ such that $\derivative{t} \vert_{t=0} u_t  = v$. For this, first we choose any $2\pi$-periodic even function $g$ such that $\int_{-\pi}^\pi g(x_1)\dd{S}(x_1) \neq 0$. 

    For $t,t'\in \mathbb{R}$, we set
    \begin{align}
    u_{t,t'}(x_1) &:= u(x_1) + tv(x_1) + t'g(x_1),\\ 
    E_{t,t'} &:= \{(x_1,x')\in \mathbb{R}^n : \abs{x'}< u_{t,t'}(x_1)\},\\
    V(t,t') &:= \abs{E_{t,t'}\cap \Omega},
    \end{align}
    with $\Omega = \{(x_1,x')\in \mathbb{R}^n  : -\pi < x_1 < \pi\}$ as usual. Notice that
    \begin{equation}
        \dv{}{t'}\Big\vert_{(t,t')=(0,0)}V(t,t') = \int_{-\pi}^\pi g(x_1)\dd{S}_{u}(x_1) \neq 0.
    \end{equation}
    By the implicit function theorem, there exists $\varepsilon >0$ and a function $\varphi \in C^1(-\varepsilon, \varepsilon)$ satisfying $\varphi(0)=0$ and $V(t,\varphi(t))=\abs{E\cap \Omega}$ for every $t\in (-\varepsilon, \varepsilon)$. On the other hand, since $\int_{-\pi}^\pi v\dd{S}_{u}=0$ and $\int_{-\pi}^\pi g\dd{S}_u \neq 0$, we have that
    \begin{equation}
        0=\dv{}{t}\Big\vert_{t=0}V(t,\varphi(t)) = \int_{-\pi}^\pi (v(x_1)+\varphi'(0)g(x_1)) \dd{S}_{u}(x_1) = \varphi'(0)\int_{-\pi}^\pi g(x_1)\dd{S}_{u}(x_1),
    \end{equation}
    hence we must have $\varphi'(0)=0$. Therefore, setting $u_t := u_{t,\varphi(t)}$, we have
    \begin{equation}
        \dv{t}\Big\vert_{t=0}u_t = \dv{}{t}\Big\vert_{t=0} u_{t,\varphi(t)} = v,
    \end{equation}
    as we wanted to see.
\end{proof}

\subsection{The case of straight cylinders in dimension 2}\label{subsec:the_case_of_straight_cylinders} When $n=2$ and $u$ is constant---that is, when $E\subset \mathbb{R}^2$ is a straight cylinder---all the expressions appearing above become particularly simple. Indeed, if $u\equiv R$, by Lemma \ref{lemma:diagonalization_L}, the operator $L_R := \frac{1}{2}D_uH_{\alpha}(R)$ diagonalizes in the basis $(e_k:=\cos(k\cdot))$ of $H^{(1+\alpha)/2}_{even}(\mathbb{T}^1)$, with corresponding eigenvalues $\mu_k(R)$ given by \eqref{eq::expression_eigenvalues}. 

Given $v\in C^{1}_{even}(\mathbb{T}^1)$, we note that 
\begin{equation}
    \int_{-\pi}^\pi v\dd{S}_R = \int_{-\pi}^\pi v(s)\dd{s} = 0
\end{equation}
if, and only if, $v$ is orthogonal to $e_0\equiv 1$ in $L^2_{even}(\mathbb{T}^1)$. Therefore, by Lemma \ref{lemma_cylindrical}, the stability of the straight cylinder of radius $R>0$ is determined by the sign of the second eigenvalue $\mu_1$ of $L_R$, as we state next. This result had already been found in \cite{CabreCsatoMas_stability}, but we include its proof here as an easy previous step to our study of the stability of the non-straight near-cylinders.
\begin{proposition}[\cite{CabreCsatoMas_stability}]\label{proposition:stab_straight_cylinders}
    Let $\alpha\in (0,1)$, $R>0$, and $L_R = \frac{1}{2}D_uH_{\alpha}(R)$. 

    Then, in the notation of Lemma \ref{lemma_cylindrical} with $u\equiv R$,
    \begin{equation}
        \mathcal{R}_R:=\inf_v\frac{\expval{v, L_R v}}{\expval{v,v}} = \mu_1(R),
    \end{equation}
    where the infimum is taken over all $v\in H^{(1+\alpha)/2}_{even}(\mathbb{T}^1)$ such that $\int_{-\pi}^\pi v(s)\dd{s} = 0$, and $\mu_1(R)$ is given by \eqref{eq::expression_eigenvalues}.

    As a consequence, the straight cylinder of radius $R$ is stable for $\mathcal{P}_{\alpha}$ if, and only if, $R\geq R_1$, where $R_1$ is defined as the unique solution of $\mu_1(R)=0$ (as in Lemma \ref{lemma:diagonalization_L}).
\end{proposition}
\begin{proof}
    The first part of the proposition is an immediate consequence of the previous observations. The last claim follows from \textit{(3)} in Lemma \ref{lemma:diagonalization_L}.
\end{proof}

\section{On the eigenvalues of the linearized NMC operator}\label{section:on_the_eigenvalues}
Let us now turn our attention to the non-straight near-cylinders of Theorem \ref{theorem_seq_bif}, 
\begin{equation}
    E_m(a) = \{(z_1,z_2)\in \mathbb{R}^2 : \abs{z_2}<w_m(a)(z_1)\},
\end{equation}
with $m\geq 1$, $a\in (-\nu_m, \nu_m)$, and $w_m(a)$ given by
\begin{equation}
    w_m(a)(s) = \gamma_m(a)R_m + a(\cos(ms)+v_m(a)(s)).
\end{equation}

In Subsection \ref{subsec:the_case_of_straight_cylinders}, we saw that the stability of the straight cylinder of radius $R>0$ as a constant-NMC surface is determined by the sign of the second eigenvalue $\mu_1(R)$ of the linearization of the NMC operator, $L_R=\frac{1}{2}D_{u}H_{\alpha}(R)$. This simple characterization of stability is possible thanks to the fact that the subspace of zero-mean functions is invariant under the operator $L_R$, and, therefore, the ``Rayleigh quotient'' $\mathcal{R}_R$ appearing in Lemma \ref{lemma_cylindrical} is an eigenvalue of $L_R$. As we anticipated in the previous section, this is not true for general constant-NMC surfaces of revolution. However, since the sets $E_m(a)$ converge to straight cylinders as $a\to 0$, we will see in the next section that the subspace of zero-mean functions is, in a way, ``\emph{almost} invariant'' under the linearized operator $L_m(a):=\frac{1}{2}D_uH_{\alpha}(w_m(a))$ for $a$ close to 0. 

Therefore, to determine the stability of the near-cylinders $E_m(a)$, it will be necessary to calculate the sign of the second eigenvalue $\mu_1(a;R_m)$ of $L_m(a)$ for $a$ close to $0$. When $m\geq 2$, since $R_m<R_1$, we know, by Proposition \ref{proposition:stab_straight_cylinders}, that the straight cylinder of radius $R_m$ is unstable. As a consequence, it will be easy to see that the bifurcated near-cylinders $E_m(a)$ are all unstable for $m\geq 2$ and $a$ close to 0. However, when $m=1$, since the straight cylinder of radius $R_1$ is semi-stable (since $\mu_1(R_1)=\mu_1(a=0;R_1)=0$), we will need to calculate an asymptotic expansion for the second eigenvalue $\mu_1(a;R_1)$ of $L_1(a)$ for $a$ close to 0.

For this, let us introduce some notation. We recall that, by Lemma \ref{lemma:diagonalization_L}, when $a=0$, the operator $L_1(0) = \frac{1}{2}D_u H_{\alpha}(R_1)$ diagonalizes with eigenfunctions $e_k=\cos(k\cdot)$ and eigenvalues $\mu_k$. For every $a\in (-\nu_1, \nu_1)$, we denote by $e_k(a)$ and $\mu_k(a)$ the eigenfunctions and eigenvalues of $L_1(a)$. In order to simplify the notation, we have omitted the dependence of the eigenvalues on the radius, since in this section we will only be dealing with the first bifurcaion branch, corresponding to $R=R_1$. In other words, we denote by
\begin{equation}
  \mu_k(a) := \mu_k(a; R_1)  
\end{equation}
the $k$-th eigenvalue of the operator 
\begin{equation}
    L_1(a) := \frac{1}{2}D_u H_{\alpha}(w_1(a)).
\end{equation}
Every $\mu_k(a)$ represents the continuation of the eigenvalues $\mu_k(R_1)=\mu_k(a=0; R_1)$ of $L_{R_1}$ (which, we recall, are given by \eqref{eq::expression_eigenvalues} for $R=R_1$) along the bifurcation branch.

We will prove in Subsection \ref{subsec:continuation_eigenvalues} that the $e_k(a)$ and the $\mu_k(a)$ do indeed exist, and that they are smooth functions of $a$. As we mentioned above, the goal of this section is to obtain an asymptotic expansion of $\mu_1(a)$ for $a$ close to 0. We will see later that $\mu_1(0)=\dot{\mu}_1(0)=0$. Hence, to obtain the desired expansion, we will need to compute $\ddot{\mu}_1(0)$. Here, and throughout the rest of the paper, we have denoted differentiation with respect to $a$ with a dot.

\subsection{A general expression for $\ddot{\mu}_1$ in smooth families of operators}\label{subsec:a_general_expression_for_ddot_mu} First, we derive a general formula for the second derivative of the eigenvalues of a symmetric operator that depends smoothly on some real parameter $a$. This procedure is well-known and widely used in several areas like quantum mechanics, but we include it here for completeness. We will later apply it to the particular case of the linearization of the NMC operator on the near-cylinders.

Let $\{L(a)\}_{a\in \mathbb{R}}$ be a smooth family of symmetric operators acting on some real Hilbert space $(\mathscr{H}, \expval{\cdot,\cdot})$. Assume that there exists an orthogonal basis $(e_l)_{l\geq 0}$ of $\mathscr{H}$ consisting of eigenvectors of $L(0)$ with corresponding simple eigenvalues $(\mu_l)_{l\geq 0}$.

Suppose further that $\mu_1 $ is simple,\footnote{One can derive similar expressions as the ones we will obtain here by requiring only that $\mu_1$ have finite multiplicity. However, since in our setting the operator $L(0)$ will be the linearized NMC operator on the straight cylinders, whose eigenvalues are simple, we can assume that $\mu_1$ is simple to make computations easier.} and that both $\mu_1$ and $e_1$ can be smoothly continued in a neighbourhood of $a=0$, so that the eigenvalue equation
\begin{equation}\label{eigenvalue_eq(a)}
    L(a) e_1(a) = \mu_1(a) e_1(a)
\end{equation}
holds for every $a$ in a neighbourhood of $a=0$. We normalize $e_1(a)$ so that
\begin{equation}
    \expval{e_1(a),e_1(a)} = \expval{e_1(0), e_1(0)}
\end{equation}
for all $a$ in a neighbourhood of $a=0$.

In the following, in order to simplify the notation, all expressions will be evaluated at $a=0$ unless otherwise stated. As before, we denote differentiation with respect to $a$ with a dot.

We differentiate \eqref{eigenvalue_eq(a)} with respect to $a$ to get
\begin{equation}\label{der_eigenval_eq}
    \dot{L} e_1+L \dot{e}_1 =\dot{\mu}_1 e_1+\mu_1\dot{e}_1.
\end{equation}
Taking the scalar product with $e_1$ in \eqref{der_eigenval_eq} we obtain
\begin{equation}
    \expval{e_1,\dot{L} e_1}+\expval{e_1, L\dot{e}_1} =\dot{\mu}_1 \expval{e_1,e_1}+\mu_1\expval{e_1,\dot{e_1}}.
\end{equation}
Therefore, using the symmetry of $L$ and the fact that
\begin{equation}\label{eq:dote_othogonal_e}
 \expval{e_1(a),\dot{e}_1(a)} = \frac{1}{2}\derivative{}{a}\expval{e_1(a), e_1(a)} = \frac{1}{2}\derivative{}{a}\expval{e_1(0), e_1(0)}=0,
\end{equation}
we get
\begin{equation}\label{eq:dot_mu}
    \dot{\mu}_1 = \frac{\expval{e_1,\dot{L} e_1}}{\expval{e_1,e_1}}.
\end{equation}

To find an expression for $\dot{e}_1$, we write
\begin{equation}
    \dot{e}_1 = \sum_{k\neq 1} \frac{\expval{e_k,\dot{e}_1}}{\expval{e_k,e_k}}e_k.
\end{equation}
Taking the scalar product with $e_k$ in \eqref{der_eigenval_eq} for $k\neq 1$ we get
\begin{equation}
    \expval{e_k,\dot{L} e_1}+\expval{e_k, L\dot{e}_1} =\mu_1\expval{e_k,\dot{e}_1},
\end{equation}
which, using the symmetry of $L$, yields
\begin{equation}
    \expval{e_k,\dot{L} e_1}-(\mu_1 - \mu_k) \expval{e_k, \dot{e}_1} =0.
\end{equation}
Hence,
\begin{equation}\label{eq:dot_e1}
    \dot{e}_1 = \sum_{k\neq 1} \frac{1}{\mu_1 - \mu_k}\frac{\expval{e_k, \dot{L} e_1}}{\expval{e_k,e_k}}e_k.
\end{equation}

Now we compute $\ddot{\mu}_{1}$. To do that, we differentiate \eqref{eigenvalue_eq(a)} twice to obtain
\begin{equation}
    \ddot{L}e_1+2\dot{L} \dot{e}_1+L\ddot{e}_1=\ddot{\mu}_1e_1+2\dot{\mu}_1\dot{e}_1+\mu_1\ddot{e}_1.
\end{equation}
Taking the scalar product with $e_1$ in the expression above, using the symmetry of $L$, and the fact that $\expval{e_1,\dot{e}_1}=0$, we get
\begin{equation}
    \expval{e_1, \ddot{L}e_1}+2\expval{e_1,\dot{L} \dot{e}_1} =\ddot{\mu}_1\expval{e_1,e_1},
\end{equation}
which, using expression \eqref{eq:dot_e1} for $\dot{e}_1$ and the symmetry of $\dot{L}$, yields
\begin{equation}\label{eq:ddot_mu}
    \ddot{\mu}_1 = \frac{\expval{e_1, \ddot{L}e_1}}{\expval{e_1,e_1}} +  2\sum_{k\neq 1} \frac{1}{\mu_1-\mu_k}\frac{\smabs{\expval{e_k,\dot{L}e_1}}^2}{\expval{e_k,e_k}\expval{e_1,e_1}}.
\end{equation}

We remark that the right-hand side of expressions \eqref{eq:dot_mu}, \eqref{eq:dot_e1}, and \eqref{eq:ddot_mu} only involve the $e_k$ and $\mu_k$ at $a=0$. On the other hand, to define $\ddot{\mu}_1(0)$ we require that the continuations $\mu_1(a)$ and $e_1(a)$ of $\mu_1=\mu_1(a=0)$ and $e_1=e_1(a=0)$ exist and are sufficiently smooth in a neighbourhood of $a=0$. The existence of such continuations, that the operator $L(a)$ is symmetric for every $a$, and that $(e_k)$ be a basis of $\mathscr{H}$ are all the assumptions that we need to derive these expressions. In particular, we do \textit{not} assume that all the $\mu_k$ and $e_k$ exist in a neighbourhood of $a=0$ that is common to all $k\geq 0$. As explained in Remark \ref{remark:weak_operator}, it is unclear how one could obtain such common neighbourhood in our setting.

\subsection{The case of the linearized NMC operator}\label{subsec:case_of_lin_NMC}
Let us now turn to the particular case of
\begin{equation}
    L(a) :=L_1(a)= \frac{1}{2}D_uH_{\alpha}(w_1(a)).
\end{equation}
(We have dropped the subscript 1 from the operator $L_1(a)$ to avoid overloading the notation). Using Lemma \ref{NLMCLemma} and a simple change of variables we can write $L(a)$ as
\begin{multline}\label{def_L(a)}
    L(a)\psi = \int_{\mathbb{R}} (\psi(s)-\psi(t)) K\left(t, w_1(a)(s)-w_1(a)(t)\right)\dd{t} \\ 
    -\int_{\mathbb{R}} (\psi(s)+\psi(t)) K\left(t, w_1(a)(s)+w_1(a)(t)\right)\dd{t}.
\end{multline}
We will now show that the operator $L(a)$ is symmetric in $L^2_{even}(\mathbb{T}^1)$. For every $a\in (-\nu_1, \nu_1)$ and every $s,t \in \mathbb{R}$, we denote
\begin{equation}
    \mathcal{K}_{-}(a)(s,t) := K(s-t,w_1(a)(s)-w_1(a)(t))= \frac{A(a)(s,t)}{\abs{s-t}^{2+\alpha}},
\end{equation}
with
\begin{equation}
    A(a)(s,t) := \left(1+ \left(\frac{w_1(a)(s)-w_1(a)(t)}{s-t}\right)^2\right)^{-\frac{2+\alpha}{2}}
\end{equation}
and
\begin{equation}
    \mathcal{K}_{0}(a)(s,t) := K(s-t,w_1(a)(s)+w_1(a)(t)) = \left((s-t)^2 + (w_1(a)(s)+w_1(a)(t))^2\right)^{-\frac{2+\alpha}{2}}.
\end{equation}
The operator $L(a)$ acting on a function $v\in C^{1,\beta}_{even}(\mathbb{T}^1)$ can then be expressed as
\begin{equation}
    L(a)v(s) = \int_{\mathbb{R}} (v(s)-v(t))\mathcal{K}_{-}(a)(s,t) \dd{t} - \int_{\mathbb{R}} (v(s)+v(t))\mathcal{K}_0(a)(s,t) \dd{t},
\end{equation}
where, as always, the first integral should be understood in the principal value sense.

For every $a\in (-\nu_1, \nu_1)$ and every $s,t \in \mathbb{R}$, the kernels $\mathcal{K}_{-}$ and $\mathcal{K}_0$ satisfy
\begin{equation}\label{eq:symmetry_cal_K}
    \mathcal{K}_{-}(a)(s,t) = \mathcal{K}_{-}(a)(t,s),
\end{equation}
and, since $w_1(a)$ is $2\pi$-periodic and even,
\begin{equation}\label{eq:evenness_cal_K}
    \mathcal{K}_{-}(a)(-s,-t) = \mathcal{K}_{-}(a)(s,t)
\end{equation}
and
\begin{equation}\label{eq:periodicity_cal_K}
    \mathcal{K}_{-}(a)(s+2\pi, t+2\pi) = \mathcal{K}_{-}(a)(s,t),
\end{equation}
and similarly for $\mathcal{K}_0$.

The following lemma states that the operator $L(a)$ is symmetric in the space $L^2_{even}(\mathbb{T}^1)$. Its proof is standard and follows as an easy consequence of the properties above and Fubini's theorem. Since we have not easily found a reference for operators of this kind in the periodic setting, we provide it for completeness.
\begin{lemma}\label{lemma:L(a)_symmetric}
    For every $a\in (-\nu_1, \nu_1)$, the operator $L(a)$ is symmetric in $L^2_{even}(\mathbb{T}^1)$.
\end{lemma}
\begin{proof}
Using \eqref{eq:evenness_cal_K}, \eqref{eq:periodicity_cal_K}, and a simple change of variables, it is easy to see that $L(a)$ maps $2\pi$-periodic even functions to $2\pi$-periodic even functions.

Now, to see that $L(a)$ is symmetric, we proceed as follows. For every $u,v \in C^\infty (\mathbb{T}^1)$, we have
\begin{align}
    & \begin{multlined}[b][0.65\displaywidth]
      \expval{u,L(a)v} = \int_{-\pi}^\pi \dd{s} \int_{\mathbb{R}}\dd{t} u(s) (v(s)-v(t))\mathcal{K}_{-}(a)(s,t) \\ - \int_{-\pi}^\pi \dd{s} \int_{\mathbb{R}}\dd{t} u(s) (v(s)+v(t))\mathcal{K}_{0}(a)(s,t)
    \end{multlined}\label{eq:first_expression_u_dot_Lav} \\
    &\begin{multlined}[t][0.65\displaywidth]
       \hphantom{ \expval{u,L(a)v}} = \sum_{k\in \mathbb{Z}} \int_{-\pi}^\pi \dd{s}  \int_{(2k-1)\pi}^{(2k+1)\pi}\dd{t} u(s)(v(s)-v(t))\mathcal{K}_{-}(a)(s,t) \\ - \int_{-\pi}^\pi \dd{s} \int_{(2k-1)\pi}^{(2k+1)\pi}\dd{t} u(s) (v(s)+v(t))\mathcal{K}_{0}(a)(s,t).
    \end{multlined}
\end{align}
For every $k\in \mathbb{Z}$, we make the change of variables $(s,t) = (\bar{s}+2\pi k,\bar{t}+2\pi k)$. Then, \eqref{eq:periodicity_cal_K} yields
\begin{align}
    \expval{u,L(a)v} &= \begin{multlined}[t][0.65\displaywidth]
        \sum_{k\in \mathbb{Z}}  \int_{-(1+2k)\pi}^{(1-2k)\pi}\dd{\bar{s}}  \int_{-\pi}^\pi\dd{\bar{t}} u(\bar{s})(v(\bar{s})-v(\bar{t}))\mathcal{K}_{-}(\bar{s},\bar{t}) \\ - \int_{-(1+2k)\pi}^{(1-2k)\pi}\dd{\bar{s}}  \int_{-\pi}^\pi\dd{\bar{t}} u(\bar{s})(v(\bar{s})+v(\bar{t}))\mathcal{K}_{0}(\bar{s},\bar{t})
    \end{multlined}\\
        & \begin{multlined}[b][0.65\displaywidth]
      =\int_{\mathbb{R}}\dd{\bar{s}}  \int_{-\pi}^\pi\dd{\bar{t}} u(\bar{s})(v(\bar{s})-v(\bar{t}))\mathcal{K}_{-}(\bar{s},\bar{t}) \\ - \int_{\mathbb{R}}\dd{\bar{s}} \int_{-\pi}^\pi\dd{\bar{t}}u(\bar{s})  (v(\bar{s})+v(\bar{t}))\mathcal{K}_{0}(\bar{s},\bar{t})
    \end{multlined}\\
    & \begin{multlined}[b][0.65\displaywidth]
       =\int_{-\pi}^\pi\dd{\bar{t}} \int_{\mathbb{R}}\dd{\bar{s}}u(\bar{s})(v(\bar{s})-v(\bar{t}))\mathcal{K}_{-}(\bar{s},\bar{t}) \\ - \int_{-\pi}^\pi\dd{\bar{t}}\int_{\mathbb{R}}\dd{\bar{s}}  u(\bar{s}) (v(\bar{s})+v(\bar{t}))\mathcal{K}_{0}(\bar{s},\bar{t})
    \end{multlined}\\
    &\begin{multlined}[b][0.65\displaywidth]
      =\int_{-\pi}^\pi\dd{\bar{s}} \int_{\mathbb{R}}\dd{\bar{t}}u(\bar{t})(v(\bar{t})-v(\bar{s}))\mathcal{K}_{-}(\bar{t},\bar{s}) \\ -  \int_{-\pi}^\pi\dd{\bar{s}} \int_{\mathbb{R}}\dd{\bar{t}}u(\bar{t}) (v(\bar{t})+v(\bar{s}))\mathcal{K}_{0}(\bar{t},\bar{s})
    \end{multlined}
    \\
    &\begin{multlined}[b][0.65\displaywidth]
     = -  \int_{-\pi}^\pi\dd{\bar{s}} \int_{\mathbb{R}}\dd{\bar{t}}u(\bar{t})(v(\bar{s})-v(\bar{t}))\mathcal{K}_{-}(\bar{s},\bar{t}) \\ -  \int_{-\pi}^\pi\dd{\bar{s}} \int_{\mathbb{R}}\dd{\bar{t}}u(\bar{t}) (v(\bar{s})+v(\bar{t}))\mathcal{K}_{0}(\bar{s},\bar{t}),
    \end{multlined}\label{eq:second_expression_u_dot_Lav}
\end{align}
where we have used \eqref{eq:symmetry_cal_K} in the last equality. We note that in the third equality we have used Fubini's theorem to exchange the order of integration. To be completely rigorous, since the integral in $\bar{t}$ involving $\mathcal{K}_-$ is a principal value, the case $k=0$ should be treated separately. It is easy to check that doing this one obtains the same final expression. 

Adding \eqref{eq:second_expression_u_dot_Lav} and \eqref{eq:first_expression_u_dot_Lav} we conclude that
\begin{multline}\label{eq:bilinear_form_La}
    \expval{u,L(a)v}  = 
        \frac{1}{2}\int_{-\pi}^\pi \dd{s} \int_{\mathbb{R}}\dd{t} (u(s)-u(t)) (v(s)-v(t))\mathcal{K}_{-}(a)(s,t) \\ - \frac{1}{2}\int_{-\pi}^\pi \dd{s} \int_{\mathbb{R}}\dd{t} (u(s)+u(t)) (v(s)+v(t))\mathcal{K}_{0}(a)(s,t) = \expval{v, L(a)u},
\end{multline}
as we wanted to see.
\end{proof}
Recall that, by Lemma \ref{lemma:diagonalization_L}, when $a=0$, the operator $L(0)$ diagonalizes in the Fourier basis of cosines, $e_k=\cos(k\cdot)$, with corresponding eigenvalues $\mu_k$ given by \eqref{eq::expression_eigenvalues} for $R=R_1$ and $k\geq 0$. As we explained at the beginning of this section, we denote by $\mu_1(a)$ and $e_1(a)$ the second eigenvalue and eigenfunction of $L(a)$, respectively, which are the continuation of $e_1=e_1(a=0)$ and $\mu_1=\mu_1(a=0;R_1)$ along the bifurcation. We normalize $e_1(a)$ so that, denoting by $\expval{\cdot,\cdot}$ the scalar product of $L^2_{even}(\mathbb{T}^1)$,
\begin{equation}
    \expval{e_1(a), e_1(a)} = \expval{e_1(0), e_1(0)} = \int_{-\pi}^\pi \cos(s)^2 \dd{s} = \frac{\pi}{2}.
\end{equation}

Applying the formulae derived in the previous subsection to the operator $L(a)$ we obtain Proposition \ref{proposition:dotmu_ddotmu_dote1} below. Its proof consists of lengthy computations (all of which are equalities involving mainly trigonometric identities), and it is included in Subsection~\ref{subsec:explicit_rayleigh}. We note that even though expression \eqref{eq:ddot_mu} for $\ddot{\mu}_1$ contains infinitely many terms, in our setting on near-cylinders, all of them vanish except for the three appearing below. These three terms will be made even more explicit in Subsection \ref{subsec:explicit_rayleigh}. They involve some Fourier coefficients of the kernel $K(z)$, which will later be expressed in terms of standard special functions in Subsection \ref{subsec:asymptotic_alpha_1}. Such expression will be crucial to determine, in Section \ref{sec:stability_of_near_cyls}, the instability of the first branch of bifurcated near-cylinders, our main result.
\begin{proposition}\label{proposition:dotmu_ddotmu_dote1}
    Let $\alpha\in (0,1)$, and let $L(a)$ be as in \eqref{def_L(a)}. Then,
    \begin{equation}
        \dot{\mu}_1(0) = 0, \quad \ddot{\mu}_1(0) = \left(\frac{\expval{e_1, \ddot{L}e_1}}{\expval{e_1,e_1}} - \frac{2}{\mu_0}\frac{\smabs{\expval{e_0,\dot{L}e_1}}^2}{\expval{e_0,e_0}\expval{e_1,e_1}}- \frac{2}{\mu_2}\frac{\smabs{\expval{e_2,\dot{L}e_1}}^2}{\expval{e_2,e_2}\expval{e_1,e_1}}\right)\Bigg\vert_{a=0},
    \end{equation}
    and
    \begin{equation}
        \dot{e}_1(0) = \left(-\frac{1}{\mu_0}\frac{\expval{e_0,\dot{L}e_1}}{\expval{e_0,e_0}\expval{e_1,e_1}}e_0 - \frac{1}{\mu_2}\frac{\expval{e_2,\dot{L}e_1}}{\expval{e_2,e_2}\expval{e_1,e_1}}e_2\right)\Bigg\vert_{a=0},
    \end{equation}
    where $\expval{\cdot,\cdot}$ denotes the scalar product of $L^2_{even}(\mathbb{T}^1)$, $e_k = \cos(k\cdot)$, and the eigenvalues $\mu_k$ are given by \eqref{eq::expression_eigenvalues} for $R=R_1$.
\end{proposition}

\subsection{Continuation of the eigenvalues}\label{subsec:continuation_eigenvalues}
In this subsection, we prove that $\mu_1(a)$ and $e_1(a)$, which are an eigenvalue-eigenvector pair for $L(a)$, exist and are smooth functions of $a$ in a neighbourhood of $a=0$. Before we do this, we make the following observation, that we include in a remark for future reference.

\begin{remark}\label{remark:weak_operator}
 The operator $L(a)$ is defined most naturally as an operator between the spaces $C^{1,\beta}_{even}(\mathbb{T}^{1})$ and $C^{\beta- \alpha}_{even}(\mathbb{T}^1)$, since it arises from the application of the implicit function theorem to the NMC operator between these spaces. We call this the \textit{strong operator}. On the other hand, since $L(a)$ is a symmetric operator in $L^2_{even}(\mathbb{T}^1)$, one can also define a ``weak operator'' between $H^{(1+\alpha)/2}_{even}(\mathbb{T}^1)$ and $L^2_{even}(\mathbb{T}^1)$ via \eqref{eq:bilinear_form_La} and Lax-Milgram's theorem. Hence, by standard compactness arguments, one can show that, for every $a$, the operator $L(a)$ diagonalizes in an orthogonal basis $e_k(a)$ of $L^2_{even}(\mathbb{T}^1)$ with eigenvalues $\mu_k(a)$. At the same time, the implicit function theorem applied to the strong operator in the Hölder spaces $C^{1,\beta}_{even}(\mathbb{T}^1)$ and $C^{\beta- \alpha}_{even}(\mathbb{T}^1)$ gives the existence and smoothness with respect to $a$ of $(\mu_k(a), e_k(a))$ for every $k\geq 0$ in a neighbourhood of $a=0$ which may depend on $k$.

 By lack of regularity results (as we explain next), we cannot guarantee that the weak eigenfunctions are also strong eigenfunctions of $L(a)$. As a consequence, we do not have the Rayleigh quotient characterization for the strong eigenvalues. Therefore, when proving Lemma \ref{lemma:lower_bound_rayleigh} below, the second equality in \eqref{eq:proof_of_lemma_impossible} will not be guaranteed to hold, and thus the lemma will require an alternative proof. It will use the specific form of the integro-differential operator $L(a)$ and properties of its kernel.

 To give a hint on the difficulties encountered regarding the regularity of the weak eigenfunctions, recall that they are obtained through a weak formulation in the spaces $H^{(1+\alpha)/2}_{even}(\mathbb{T}^1)$ and $L^2_{even}(\mathbb{T}^1)$. To show that they are, in fact, strong eigenfunctions, we would need to show that they are of class $C^{1,\beta}$. Notice that this is the maximum regularity that one can expect from a divergence-form integro-differential operator of order $1+\alpha>1$ with $C^{\beta}$ coefficients, like our operator $L(a)$. Such regularity gain seems, according to some experts, not to be available for non-translation-invariant operators. In this direction, even to know whether $L(a)$ maps $H^{1+\alpha}(\mathbb{T}^1)$ into $L^2(\mathbb{T}^1)$ seems to be an open question. Note that when $a=0$ the operator $L(0)$ is the fractional Laplacian $(-\Delta)^{(1+\alpha)/2}$ plus a 0-order term. In this case, $L(0)$ does indeed map $H^{1+\alpha}$ into $L^2$, as can be shown using Fourier series thanks to the translation-invariance.
\end{remark}

For every $k\geq 1$, we will apply the implicit function theorem to the eigenvalue equation
\begin{equation}\label{eq:eigenvalue_problem}
    \mathcal{L}_k:(-\nu_1, \nu_1)\times \mathbb{R} \times X_k \to C^{\beta- \alpha}_{even}(\mathbb{T}^1), \quad \mathcal{L}_k(a, \mu, \varphi) = L(a)(e_k+\varphi)- \mu (e_k+\varphi) = 0,
\end{equation}
where $\nu_1$ is given by Theorem \ref{theorem_seq_bif},
\begin{equation}
    X_k := \{v\in C^{1,\beta}_{even}(\mathbb{T}^1) : \expval{v, e_k} = 0\},
\end{equation}
and $L(a)$ is defined in \eqref{def_L(a)}. By construction, we know that $\mathcal{L}_k(0, \mu_k, 0)=0$, where $\mu_k=\mu_k(0)$ is the $k$-th eigenvalue of $L(0)$, as given by \eqref{eq::expression_eigenvalues} for $R=R_1$. To apply the implicit function theorem, we need to check that the operator $\mathcal{L}_k$ is of class at least $C^2$ (since we want to calculate $\ddot{\mu}_1$) in a neighbourhood of $(0,\mu_k,0)$, and that its differential,
\begin{equation}
\begin{array}{cccl}
    D_{(\mu, \varphi)}\mathcal{L}_k(0,\mu_k, 0) : &\mathbb{R}\times X_k &\to &C^{\beta- \alpha}_{even}(\mathbb{T}^1)\\ 
    &(s, \psi)&\mapsto  &-se_k+(L(0) - \mu_k)\psi,
\end{array}
\end{equation}
is invertible. The proof of these facts is contained in the following proposition.
\begin{proposition}\label{prop:continuation_eigenvalues}
    Let $\alpha\in (0,1)$ and $\alpha < \beta< \min\{1, 2\alpha + \frac{1}{2}\}$. For every $k\geq 1$, the operator $\mathcal{L}_k$, as defined in \eqref{eq:eigenvalue_problem}, is of class $C^{\infty}$ in a neighbourhood of $(0,\mu_k,0)$, and its differential,
\begin{equation}
    \begin{array}{cccl}
    D_{(\mu, \varphi)}\mathcal{L}_k(0,\mu_k, 0) : &\mathbb{R}\times X_k &\to &C^{\beta- \alpha}_{even}(\mathbb{T}^1)\\ 
    &(s, \psi)&\mapsto  &-se_k+(L(0) - \mu_k)\psi,
\end{array}
\end{equation}
is invertible. 

As a consequence, for every $k\geq 1$ there exists $\tilde{\nu}_k\in (0, \nu_1)$, depending only on $k$, $\alpha$, and $\beta$, and a $C^\infty$ curve 
\begin{equation}
    \begin{array}{ccc}
    (-\tilde{\nu}_k, \tilde{\nu}_k) & \to &\mathbb{R}\times X_k\\ 
    a&\mapsto  &(\mu_k(a), \varphi_k(a)),
\end{array}
\end{equation}
such that, for every $a\in (-\tilde{\nu}_k, \tilde{\nu}_k)$,
\begin{equation}
    L(a)(e_k + \varphi_k(a)) = \mu_k(a)(e_k + \varphi_k(a)).
\end{equation}
That is, $e_k+\varphi_k(a)$ is an eigenfunction of $L(a)$ with eigenvalue $\mu_k(a)$. 

Moreover, $(\mu_k(a), \varphi_k(a))\to (\mu_k, 0)$ in $\mathbb{R}\times C^{1,\beta}_{even}(\mathbb{T}^1)$ as $a\to 0$.
\end{proposition}
\begin{proof}
Since, by Theorem \ref{theorem_seq_bif}, the map $a\mapsto w_1(a)$ is smooth and $w_1(a)>0$ for $\abs{a}$ small enough, we know by Lemma \ref{lemma_H_Cinfty} that $L(a)=\frac{1}{2}D_uH(w_1(a))$ is of class $C^\infty$ in a neighbourhood of $a=0$. Since the operator $\mathcal{L}_k$ is linear in $\mu$ and $\varphi$, we conclude that $\mathcal{L}_k$ is also of class $C^\infty$ in a neighbourhood of $(0, \mu_k, 0)$.

It remains to prove that the differential $D_{(\mu, \varphi)} \mathcal{L}_k(0, \mu_k, 0)$ is invertible. For every $k \geq 1$, this follows easily from the fact that, since the eigenvalue $\mu_k$ is simple, the kernel of $L(0) - \mu_k$ is spanned by $e_k$. Therefore, repeating the argument at the end of Step 1 in the proof of Theorem~\ref{theorem_seq_bif}, we deduce that $D_{(\mu, \varphi)}\mathcal{L}_k$ is invertible at $(0, \mu_k, 0)$.

By the implicit function theorem applied to $\mathcal{L}_k$ at the point $(0,\mu_k, 0)$, we conclude that there exists a smooth curve $\{(\mu_k(a), \varphi_k(a)) : a\in (-\tilde{\nu}_k , \tilde{\nu}_k)\}$ of solutions of $\mathcal{L}_k(a, \mu_k(a), \varphi_k(a))=~0$ satisfying the desired properties.
\end{proof}
As we did in the previous subsection, we normalize the eigenvectors of $L(a)$ so that they have the same $L^2_{even}(\mathbb{T}^1)$ norms as the $e_k$. That is, throughout the rest of the paper, we set
\begin{equation}
    e_k(a) :=  \frac{\expval{e_k(0),e_k(0)}}{\expval{e_k(0),e_k(0)}+\expval{\varphi_k(a),\varphi_k(a)}}(e_k(0)+\varphi_k(a)),
\end{equation}
with $\varphi_k(a)$ given by Proposition \ref{prop:continuation_eigenvalues} above.

\section{On the stability of the near-cylinders}\label{sec:stability_of_near_cyls}

In this section, we prove the instability of the near-cylinders of Theorem \ref{theorem_seq_bif}.
To do this, we want to apply Lemma \ref{lemma_cylindrical} to the operator $L(a)~=~\frac{1}{2}D_uH_{\alpha}(w_1(a))$. Now, $n=2$ and $u=w_1(a)$. Note that, since $n=2$, the surface measure $\dd{S_u}$ is proportional to the one-dimensional Lebesgue measure. Therefore, the infimum in the definition of the Rayleigh quotient \eqref{def:rayleigh_quot} is taken over the subspace
\begin{equation}
    V_0(0)^\perp : =\{v\in H^{(1+\alpha)/2}_{even}(\mathbb{T}^1) : \textstyle \int_{-\pi}^\pi v(s)\dd{s} = 0\}.
\end{equation}
(The justification for this notation will be made clear in the following subsection). The main obstacle that we will encounter in calculating this infimum is that, as we mentioned before, for $a\neq 0$ (that is, when the set $E_m(a)$ is not a straight cylinder), the subspace $V_0(0)^\perp$ is \textit{not} invariant under $L(a)$. As a consequence, the Rayleigh quotient\footnote{Following Lemma \ref{lemma_cylindrical}, this Rayleigh quotient should be denoted by $\mathcal{R}_{w_1}$. We have chosen to write it as $\mathcal{R}_1$ to simplify the notation.}
\begin{equation}\label{eq:rayleigh_for_first_branch}
    \mathcal{R}_1(a) := \inf_{v\in V_0(0)^\perp} \frac{\expval{v, L(a)v}}{\expval{v,v}}
\end{equation}
appearing in Lemma \ref{lemma_cylindrical} is \emph{not} an eigenvalue of $L(a)$ when $a\neq 0$, and, therefore, this will create difficulties. It would indeed be an eigenvalue if the infimum were taken over the space
\begin{equation}
    V_0(a)^\perp
\end{equation}
instead of
\begin{equation}
    V_0(0)^\perp.
\end{equation}
As before, we denote by $\expval{\cdot, \cdot}$ the scalar product of $L^2_{even}(\mathbb{T}^1)$.

\subsection{An asymptotic formula for $\mathcal{R}_1(a)$}
Given $k_1,\dots,k_l\geq 1$ and $a\in (-\nu_1, \nu_1)$, we denote
\begin{equation}
    V_{k_1,\dots,k_l}(a) := \mathrm{span}\{e_{k_1}(a),\dots, e_{k_l}(a)\}
\end{equation}
and, with a small abuse of notation,\footnote{Indeed, $V_{k_1,\dots,k_l}(a)^\perp$ is not the orthogonal complement of $V_{k_1,\dots,k_l}(a)$ in $H^{(1+\alpha)/2}_{even}(\mathbb{T}^1)$, but the intersection with $H^{(1+\alpha)/2}_{even}(\mathbb{T}^1)$ of its orthogonal complement in $L^2_{even}(\mathbb{T}^1)$. However, since $V_{k_1,\dots,k_l}(a)\subset H^{(1+\alpha)/2}_{even}(\mathbb{T}^1)$, $V_{k_1,\dots,k_l}(a)^\perp$ is still a complement in $H^{(1+\alpha)/2}_{even}(\mathbb{T}^1)$. That is, the decomposition $v=v_{k_1,\dots,k_l} + v_{\perp}$ in $L^2_{even}(\mathbb{T}^1)$ of every $v\in H^{(1+\alpha)/2}_{even}(\mathbb{T}^1)$ holds also in $H^{(1+\alpha)/2}_{even}(\mathbb{T}^1)$, because $v_{\perp} = v-v_{k_1,\dots,k_l}\in H^{(1+\alpha)/2}_{even}(\mathbb{T}^1)$.}
\begin{equation}
    V_{k_1,\dots,k_l}(a)^\perp := \{ v\in H^{(1+\alpha)/2}_{even}(\mathbb{T}^1) : \expval{v,e_{k_1}(a)}=\dots =\expval{v, e_{k_l}(a)}=0 \},
\end{equation}
where $e_{k_i}(a)$ are eigenvectors of $L(a)$, as they were described in Subsection \ref{subsec:continuation_eigenvalues}. Here, and throughout the whole text, we have denoted by $H^{(1+\alpha)/2}_{even}(\mathbb{T}^1)$ the subspace of the fractional Sobolev space $H^{(1+\alpha)/2}(-\pi,\pi)$ consisting of $2\pi$-periodic even functions. Notice that the subspace $V_0(0)^\perp$ defined before is precisely the subspace of functions in $H^{(1+\alpha)/2}_{even}(\mathbb{T}^1)$ that are orthogonal in $L^2_{even}(\mathbb{T}^1)$ to $1=e_0(a=0)$ (which is the first eigenfunction of the operator $L(0)$).

To evaluate \eqref{eq:rayleigh_for_first_branch}, we will consider the decomposition
\begin{equation}
    H^{(1+\alpha)/2}_{even}(\mathbb{T}^1) = V_{0,1}(a)\oplus V_{0,1}(a)^\perp,
\end{equation}
and we will show that the infimum $\mathcal{R}_1$ is attained in $V_0(0)^\perp\cap V_{0,1}(a)$. This is precisely the content of the following proposition.
\begin{proposition}\label{proposition:infimum_attained}
    Let $\alpha\in (0,1)$ and $\alpha < \beta< \min\{1, 2 \alpha + \frac{1}{2}\}$. For all $\abs{a}$ small enough, depending only on $\alpha$ and $\beta$, the subspace \mbox{$V_0(0)^\perp\cap V_{0,1}(a)$} is nonzero (note that in this intersection the first subspace corresponds to $a=0$ and the second one, to $a\neq 0$), and
    \begin{equation}
        \mathcal{R}_1(a) = \inf_{v\in V_0(0)^\perp }\frac{\expval{v,L(a)v}}{\expval{v,v}}=\inf_{v\in V_0(0)^\perp\cap V_{0,1}(a)}\frac{\expval{v,L(a)v}}{\expval{v,v}}.
    \end{equation}
\end{proposition}

The proof of this proposition is an easy consequence of the following lemma.
\begin{lemma}\label{lemma:lower_bound_rayleigh}
Let $\alpha\in (0,1)$ and $\alpha < \beta< \min\{1, 2 \alpha + \frac{1}{2}\}$. For $\abs{a}$ small enough, depending only on $\alpha$ and $\beta$,
    \begin{equation}
        \sup_{v\in V_{0,1}(a)} \frac{\expval{v, L(a)v}}{\expval{v,v}} < \inf_{v\in V_{0,1}(a)^{\perp}} \frac{\expval{v, L(a)v}}{\expval{v,v}}.
    \end{equation}
\end{lemma}

This result might seem trivial at first. Indeed, if we could write the whole spectral decomposition of $L(a)$, it would be immediate to check (since here we refer to the spaces $V_{0,1}(a)$ and not $V_{0,1}(0)$) that
\begin{equation}\label{eq:proof_of_lemma_impossible}
    \sup_{v\in V_{0,1}(a)} \frac{\expval{v, L(a)v}}{\expval{v,v}} = \mu_1(a) < \mu_2(a) = \inf_{v\in V_{0,1}(a)^{\perp}} \frac{\expval{v, L(a)v}}{\expval{v,v}}.
\end{equation}
The middle inequality would follow from the facts that $\mu_1(0)<\mu_2(0)$ and that the eigenvalues depend continuously on $a$. However, as we discussed in Remark \ref{remark:weak_operator}, we \textit{cannot} do this, as we do not have the whole spectral decomposition of $L(a)$ for $a\neq 0$, and thus the second equality in the expression above need not hold. Hence, we must proceed in a different way.

We note first that the lemma holds true for $a=0$, since, by Lemma~\ref{lemma:diagonalization_L}, the operator $L(0)$ diagonalizes in the basis $(e_k(a=0) = \cos(k\cdot))$ of $H^{(1+\alpha)/2}_{even}(\mathbb{T}^1)$, with corresponding eigenvalues given by \eqref{eq::expression_eigenvalues} for $R=R_1$. On the other hand, since the subspace $V_{0,1}(a)$ is finite-dimensional, for every $a$ we have that
\begin{equation}
    \sup_{v\in V_{0,1}(a)} \frac{\expval{v, L(a)v}}{\expval{v,v}} = \mu_1(a).
\end{equation}
Then, by continuity arguments based on the particular structure of the operator $L(a)$, we will be able to show that, for $\abs{a}$ small enough,
\begin{equation}
    \mu_1(a) < \inf_{v\in V_{0,1}(a)^{\perp}} \frac{\expval{v, L(a)v}}{\expval{v,v}}.
\end{equation}

\begin{proof}[Proof of Lemma \ref{lemma:lower_bound_rayleigh}]
First, we write the operator $L(a)$ as $L(a) = L^-(a)-L^0(a)$, with
\begin{equation}
    L^{-}(a) \psi = \int_{\mathbb{R}} (\psi(s)-\psi(t))\mathcal{K}_-(a)(s,t)\dd{t} \quad \text{and}\quad L^{0}(a) \psi = \int_{\mathbb{R}} (\psi(s)+\psi(t))\mathcal{K}_0(a)(s,t)\dd{t},
\end{equation}
and $\mathcal{K}_-$ and $\mathcal{K}_0$ as defined in Subsection \ref{subsec:case_of_lin_NMC}.

Let $\varepsilon >0$. By Theorem \ref{theorem_seq_bif}, we have that $A(a)\to A(0)\equiv 1$ uniformly in $(s,t)$ as $a\to 0$. Therefore, for $\abs{a}$ small enough, independently of $v$,
\begin{equation}\label{eq:lower_bound_L-}
    \expval{v, L^-(a)v} = \expval{v, L^-(0)v} + \expval{v, (L^-(a)-L^-(0))v} \geq (1- \varepsilon) \expval{v, L^-(0)v}.
\end{equation}

Now, for $L^0$ we note that since $\mathcal{K}_0(0)$ is a bounded function, $\expval{v,L^0(0)v}$ is comparable to $\expval{v,v}$. Moreover, the map $a\mapsto L^0(a)$ is continuous when viewing $L^0(a)$ as an operator acting on $L^2_{even}(\mathbb{T}^1)$. To see this, notice that $L^0(a)$ is essentially a Hilbert-Schmidt integral operator, and use the facts that $\mathcal{K}_0(a)\to \mathcal{K}_0(0)$ as $a\to 0$, and that
\begin{equation}
  \mathcal{K}_0(a)(s,t) \leq C(a) \min\{1, \abs{s-t}^{-2-\alpha}\}  
\end{equation}
for some constant $C(a)$ that remains bounded as $a\to 0$. Therefore, for $\abs{a}$ small enough, independently of $v$,
\begin{align}
    \expval{v, L^0(a)v} &= (1-\varepsilon)\expval{v, L^0(0)v} + \expval{v, (L^0(a)-L^0(0))v} +\varepsilon\expval{v, L^0(0)v}\\ 
    &\leq (1-\varepsilon)\expval{v, L^0(0)v} + ||L^0(a)-L^0(0)||_{L^2\to L^2} \expval{v, v} +\varepsilon ||L^0(0)||_{L^2\to L^2}\expval{v, v}\\
    &\leq (1-\varepsilon)\expval{v, L^0(0)v} + \varepsilon C\expval{v, v}, \label{eq:upper_bound_L0}
\end{align}
for some $C$ depending only on $||L^0(0)||_{L^2\to L^2}$.

Using \eqref{eq:lower_bound_L-} and \eqref{eq:upper_bound_L0} we have that
\begin{align}
    \expval{v, L(a)v} &= \expval{v, L^-(a)v} -  \expval{v, L^0(a)v} \\ &\geq (1- \varepsilon) \expval{v, L^-(0)v}- (1-\varepsilon)\expval{v, L^0(0)v} - \varepsilon C \expval{v, v}\\ 
    &= (1- \varepsilon) \expval{v, L(0)v} -\varepsilon C \expval{v, v}. \label{eq:bound_La_L0}
\end{align}

Now, let $v\in V_{0,1}(a)^\perp$, and assume that $\expval{v,v}=1$. Since the functions $e_k(0)=\cos(k\cdot)$ form a basis of $H^{(1+\alpha)/2}_{even}(\mathbb{T}^1)$, we may write $v = v^0e_0(0)+v^1e_1(0)+v_\perp$, with $v_\perp \in V_{0,1}(0)^\perp$. By definition of $V_{0,1}(a)^\perp$, we must have
\begin{equation}
    0 = \int_{-\pi}^\pi v(s)e_0(a)\dd{s} = \int_{-\pi}^\pi v(s)(e_0(0)+(e_0(0)-e_0(a)))\dd{s}.
\end{equation}
Therefore,
\begin{equation}
    \smabs{v^0} =  \frac{1}{\expval{e_0,e_0}}\abs{\int_{-\pi}^\pi v(s)(e_0(a)(s)-e_0(0)(s)) \dd{s}} \leq \frac{1}{\pi}\norm{e_0(a)-e_0(0)}_{L^2_{even}(\mathbb{T}^1)} .
\end{equation}
Since, by Proposition \ref{prop:continuation_eigenvalues}, we have that $e_0(a)\to e_0(0)$ uniformly as $a\to 0$, we deduce that, for every $\varepsilon>0$, picking $\abs{a}$ small enough, independently of $v$,
\begin{equation}\label{eq:bound_v0}
    \smabs{v^0}^2 \leq \frac{4}{\pi^2}\norm{e_0(a)-e_0(0)}_{L^2_{even}(\mathbb{T}^1)}^2 \leq \varepsilon.
\end{equation}
The same argument replacing $e_0$ by $e_1$ yields
\begin{equation}
    \smabs{v^1}^2 \leq \frac{1}{\pi^2}\norm{e_0(a)-e_0(0)}_{L^2_{even}(\mathbb{T}^1)}^2 \leq \varepsilon.
\end{equation}

Now, recall that, since $L(0)$ diagonalizes, we have that
\begin{equation}
    \mu_2(0) = \inf_{v\in V_{0,1}(0)^\perp} \frac{\expval{v,L(0)v}}{\expval{v,v}}.
\end{equation}
As a consequence, using this, \eqref{eq:bound_La_L0}, and \eqref{eq:bound_v0}, picking $\varepsilon$ small enough (independently of $v$), we have that, for $\abs{a}$ small enough, since $\mu_0(0)<0=\mu_1(0)<\mu_2(0)$,
\begin{align}
    \expval{v, L(a)v} &\geq (1-\varepsilon)\expval{v, L(0)v} - \varepsilon C \\ 
    & \geq (1-\varepsilon)\left\{\smabs{v^0}^2 \mu_0(0) \expval{e_0,e_0}+\smabs{v^1}^2 \mu_1(0) \expval{e_1,e_1} + \mu_2(0)\expval{v_{\perp},v_{\perp}} \right\}-\varepsilon C \\ 
    &\geq (1-\varepsilon)\left\{\varepsilon \mu_0(0)\expval{e_0,e_0} + \mu_2(0)(1-2\varepsilon)\right\} - \varepsilon C \\
    &\geq \frac{1}{2}\mu_2(0)\label{eq:lower_bound_mu2_half}
\end{align}
for all $v\in V_{0,1}(a)^\perp$ such that $\expval{v,v}=1$.

Finally, since, by continuity we have that, for $\abs{a}$ small enough,
\begin{equation}
    \mu_1(a)<\frac{1}{2}\mu_2(0),
\end{equation}
using this and \eqref{eq:lower_bound_mu2_half} we deduce that
\begin{equation}
    \sup_{v\in V_{0,1}(a)}\frac{\expval{v,L(a)v}}{\expval{v,v}} = \mu_1(a)< \inf_{v\in V_{0,1}(a)^\perp} \frac{\expval{v,L(a)v}}{\expval{v,v}},
\end{equation}
as we wanted to see.
\end{proof}
\begin{proof}[Proof of Proposition \ref{proposition:infimum_attained}]
    Let $v\in V_{0}(0)^\perp$, and assume that $\expval{v,v}=1$. We write $v = v_{0,1}+v_\perp$, with $v_{0,1}\in V_{0,1}(a)$ and $v_\perp \in V_{0,1}(a)^{\perp}$. Then, since the subspace $V_{0,1}(a)$ is invariant under $L(a)$, we have that
    \begin{align}
        {\expval{v,L(a)v}} &= {\expval{v_{0,1},L(a)v_{0,1}}} + {\expval{v_{\perp},L(a)v_{\perp}}}\\ &= 
        \frac{\expval{v_{0,1},L(a)v_{0,1}}}{\expval{v_{0,1}, v_{0,1}}} \expval{v_{0,1}, v_{0,1}} + \frac{\expval{v_{\perp},L(a)v_{\perp}}}{\expval{v_{\perp}, v_{\perp}}} \expval{v_{\perp}, v_{\perp}}.\label{eq:average_quotients}
    \end{align}
    By Lemma \ref{lemma:lower_bound_rayleigh}, the first quotient in \eqref{eq:average_quotients} is smaller than the second one. In addition, \eqref{eq:average_quotients} is a convex combination, since $\expval{v_{0,1}, v_{0,1}} + \expval{v_{\perp}, v_{\perp}}=\expval{v,v}=1$. As a consequence, we have that
    \begin{equation}
        \frac{\expval{v_{0,1},L(a)v_{0,1}}}{\expval{v_{0,1}, v_{0,1}}}  \leq {\expval{v,L(a)v}} \leq \frac{\expval{v_{\perp},L(a)v_{\perp}}}{\expval{v_{\perp}, v_{\perp}}}
    \end{equation}
    for all $v\in V_{0}(0)^\perp$ such that $\expval{v,v}=1$, for all $\abs{a}$ small enough. Therefore, $\expval{v,L(a)v}$ will attain its infimum whenever $v=v_{0,1}$, provided that the subspace $V_{0}(0)^\perp \cap V_{0,1}(a)$ is nonzero.

    It only remains to see that the subspace $V_{0}(0)^\perp \cap V_{0,1}(a)$ is nonzero. In fact, we will show that it is one-dimensional. Indeed, every $v\in V_0(0)^\perp\cap V_{0,1}(a)$ is of the form $v=v^0(a)e_0(a)+v^1(a)e_1(a)$, and it must satisfy
    \begin{equation}
        0 = \int_{-\pi}^{\pi} v(s)\dd{s} = v^0(a) \int_{-\pi}^\pi e_0(a)(s)\dd{s} + v^1(a) \int_{-\pi}^{\pi} e_1(a)(s)\dd{s}.
    \end{equation}
    Therefore,
    \begin{equation}
        v^0(a) = -v^1(a) \left(\frac{\int_{-\pi}^{\pi} e_1(a)(s)\dd{s}}{\int_{-\pi}^{\pi}  e_0(a)(s)\dd{s}}\right) =: -v^1(a) Q(a).
    \end{equation}
    Note that $Q(a)$ is well-defined, since $\int_{-\pi}^\pi e_0(0)(s)\dd{s}=\int_{-\pi}^\pi 1\dd{s}\neq 0$ and $e_0(a)$ depends continuously on $a$. As a consequence, we have that
    \begin{equation}\label{eq:one_dimensional}
        v = \lambda \left( e_1(a) - Q(a)e_0(a) \right),
    \end{equation}
    where $\lambda = v^1(a)\in \mathbb{R}$. Hence, $V_{0}(0)^\perp \cap V_{0,1}(a)$ is one-dimensional.
\end{proof}

We are now ready to state and prove the key ingredient in the proof of Theorem \ref{theorem_stability_near_cylinders}, Proposition \ref{proposition:rayleigh_asymptotic} below, which gives an asymptotic expression for $\mathcal{R}_1(a)$ when $a$ is close to 0. The sign of this expression for $a\neq 0$ will determine the stability of the first branch near-cylinders.
\begin{proposition}\label{proposition:rayleigh_asymptotic}
    For every $\alpha\in (0,1)$, we have
    \begin{equation}\label{eq:asymptotic_expression_ddot_R1}
        \mathcal{R}_1(a) :=\inf_{v\in V_0(0)^\perp} \frac{\expval{v, L(a)v}}{\expval{v,v}} = \left(\frac{1}{2}\frac{\expval{e_1, \ddot{L}e_1}}{\expval{e_1,e_1}} - \frac{1}{\mu_{2}}\frac{\smabs{\expval{e_{2},\dot{L}e_1}}^2}{\expval{e_{2},e_{2}}\expval{e_1,e_1}}\right)\Bigg\vert_{a=0} a^2 + O(a^3),
    \end{equation}
    where, we recall, $V_0(0)^\perp=\{ v \in H^{(1+\alpha)/2}_{even}(\mathbb{T}^1) : \int_{-\pi}^\pi v(s)\dd{s} = 0 \}$.
\end{proposition}
\begin{proof}
    Since $L(a)$ is symmetric in $L^2_{even}(\mathbb{T}^1)$, $e_0(a)$ and $e_1(a)$ are orthogonal in $L^2_{even}(\mathbb{T}^1)$. Hence, denoting $\gamma = \expval{e_0,e_0}/\expval{e_1,e_1}$, by Proposition \ref{proposition:infimum_attained}, using expression \eqref{eq:one_dimensional} we obtain
    \begin{equation}\label{eq:R_1_before_taylor}
        \mathcal{R}_1(a) =\mu_0(a) \frac{\gamma Q(a)^2}{1+\gamma Q(a)^2} + \mu_1(a) \frac{1}{1+ \gamma Q(a)^2}.
    \end{equation}
    
    Now, since $\int_{-\pi}^\pi e_1(0)(s)\dd{s} =\int_{-\pi}^\pi \cos(s)\dd{s}= 0$, we have that $Q(0)=0$. On the other hand, recall that, by Proposition \ref{proposition:dotmu_ddotmu_dote1}, $\mu_1(0)=\dot{\mu}_1(0)=0$. Therefore, expanding \eqref{eq:R_1_before_taylor} in $a$ we see that
    \begin{align}
        \mathcal{R}_1(a)&=\gamma (\mu_0(0)+O(a))(\dot{Q}(0)a+O(a^2))^2 + \frac{1}{2}\ddot{\mu}_1(0) a^2 + O(a^3)\\ 
        &= \left(\gamma\mu_0(0) \dot{Q}(0)^2 + \frac{1}{2}\ddot{\mu}_1(0)\right)a^2 + O(a^3).
    \end{align}

    To calculate $\dot{Q}(0)$, we must calculate $\int_{-\pi}^\pi \dot{e}_0(0)(s)\dd{s}$ and $\int_{-\pi}^\pi \dot{e}_1(0)(s)\dd{s}$. Since, by the same argument as in \eqref{eq:dote_othogonal_e}, we have that $\dot{e}_0(0) \in V_0(0)^\perp$, we deduce that
    \begin{equation}
        \int_{-\pi}^{\pi} \dot{e}_0(0)(s) \dd{s}= \expval{\dot{e}_0(0), e_0(0)} =  0.
    \end{equation}
    On the other hand, since $\int_{-\pi}^\pi e_2(0)(s)\dd{s} = \int_{-\pi}^\pi \cos(2s)\dd{s} =  0$, by Proposition \ref{proposition:dotmu_ddotmu_dote1} we have that
    \begin{align}
        \int_{-\pi}^\pi \dot{e}_1(0)(s)\dd{s} = \frac{-1}{\mu_0}\frac{\expval{e_0, \dot{L} e_1}}{\expval{e_0,e_0}}\bigg\vert_{a=0}\int_{-\pi}^\pi e_0(0)(s)\dd{s}.
    \end{align}
    Therefore,
    \begin{equation}
        \dot{Q}(0) = \frac{\int_{-\pi}^{\pi} \dot{e}_1(0)(s)\dd{s}}{\int_{-\pi}^{\pi}  e_0(0)(s)\dd{s}} = \frac{-1}{\mu_0}\frac{\expval{e_0, \dot{L} e_1}}{\expval{e_0,e_0}}\bigg\vert_{a=0}.
    \end{equation}
    Hence,
    \begin{equation}
        \mathcal{R}_1(a) =\left(
        \frac{1}{\mu_0}\frac{\smabs{\expval{e_0, \dot{L} e_1}}^2}{\expval{e_0,e_0}\expval{e_1,e_1}} + \frac{1}{2}\ddot{\mu}_1 \right)\Bigg\vert_{a=0} a^2 + O(a^3).
    \end{equation}
    Now, since by Proposition \ref{proposition:dotmu_ddotmu_dote1} we have that
\begin{equation}
    \ddot{\mu}_1(0) = \left(\frac{\expval{e_1, \ddot{L}e_1}}{\expval{e_1,e_1}} - \frac{2}{\mu_0}\frac{\smabs{\expval{e_0,\dot{L}e_1}}^2}{\expval{e_0,e_0}\expval{e_1,e_1}}- \frac{2}{\mu_{2}}\frac{\smabs{\expval{e_{2},\dot{L}e_1}}^2}{\expval{e_{2},e_{2}}\expval{e_1,e_1}}\right)\Bigg\vert_{a=0},
\end{equation}
we obtain the desired formula,
\begin{equation}
    \mathcal{R}_1(a) =\left(\frac{1}{2}\frac{\expval{e_1, \ddot{L}e_1}}{\expval{e_1,e_1}} - \frac{1}{\mu_{2}}\frac{\smabs{\expval{e_{2},\dot{L}e_1}}^2}{\expval{e_{2},e_{2}}\expval{e_1,e_1}}\right)\Bigg\vert_{a=0} a^2 + O(a^3).
\end{equation}
\end{proof}

\subsection{Instability of the near-cylinders}
By explicitly evaluating the expression for $\mathcal{R}_1(a)$ given above when $\alpha\uparrow 1$, we obtain the following proposition. Since its proof involves lengthy computations, we delay it to Subsection \ref{subsec:asymptotic_alpha_1}.
\begin{proposition}\label{proposition:asymptotic_alpha_to_1}
There exists $\alpha_0 \in (0,1)$ such that, for every $\alpha \in [\alpha_0,1)$ and every $\alpha < \beta< \min\{1, 2\alpha + \frac{1}{2}\}$, we have $\mathcal{R}_1(a)<0$ for $\abs{a}>0$ small enough (but nonzero), depending only on $\alpha$ and $\beta$.
\end{proposition}
\begin{remark}\label{remark:generalizations}
    Throughout the whole text, we have denoted the kernel in the NMC operator simply by $K$, without writing it explicitly. This is because all the results presented in this paper up to Proposition \ref{proposition:asymptotic_alpha_to_1} above can be generalized in a relatively straightforward way to a broader class of anisotropic kernels as those considered in \cite{AlcoverBruera}\footnote{Still, one would need to prove that the corresponding anisotropic NMC operator is sufficiently smooth. In \cite{AlcoverBruera} it is shown that the anisotropic NMC operator is of class $C^1$ whenever $K$ satisfies some regularity assumptions. To extend the results in this paper to the anisotropic setting, one would need to show that the anisotropic NMC operator is of class at least $C^2$.}---for example, kernels of the form $K(z) = \norm{z}_K^{-2-\alpha}$, with $\norm{\cdot}_K$ a norm in $\mathbb{R}^2$ such that $\norm{(z_1,z_2)}_K=\norm{(z_1,-z_2)}_K$, and satisfying at least one of the assumptions in \cite[Theorem 1.3]{AlcoverBruera}. However, our proof of Proposition~\ref{proposition:asymptotic_alpha_to_1} requires explicit computations that depend on the particular form of $K$, therefore it is only valid for $K(z)=\abs{z}^{-2-\alpha}$.
\end{remark}
Finally, we have all the ingredients to prove Theorem \ref{theorem_stability_near_cylinders}.
\begin{proof}[Proof of Theorem \ref{theorem_stability_near_cylinders}.]
    The first part of \textit{(1)} was already stated in Proposition \ref{proposition:stab_straight_cylinders}. For the asymptotic behaviour of $R_1$ when $\alpha \uparrow 1$, we have, by \eqref{eq:asymptotic_c_alpha_1}, that
    \begin{equation}
        R_1^{1+\alpha}\Gamma\left(\frac{1-\alpha}{2}\right) = O(1)
    \end{equation}
    when $\alpha \uparrow 1$. Using that $\Gamma(s)=O(\frac{1}{s})$ as $s\downarrow 0$, we deduce that $R_1=O(\sqrt{1-\alpha})$.

    Let us turn to \textit{(2)}. By Proposition \ref{proposition:asymptotic_alpha_to_1} and Lemma \ref{lemma_cylindrical}, we have that, for all $\alpha\geq \alpha_0$, the $E_1(a)$ are unstable for $\abs{a}>0$ small enough. Note that we must exclude the case $a=0$ since it corresponds to the straight cylinder of radius $R_1$, which is semi-stable (i.e., $\mathcal{R}_1(0)=0$).

    Finally, to prove $(3)$ we must show that the $E_m(a)$ are unstable for all $m\geq 2$ and all $\alpha\in (0,1)$ provided that $\abs{a}$ is taken small enough. To see this, for every $m\geq 2$ we will repeat the arguments used in the proofs of Proposition~\ref{proposition:infimum_attained} and Lemma~\ref{lemma:lower_bound_rayleigh} with the operator 
    \begin{equation}
        L_m(a) := \frac{1}{2}D_uH_{\alpha}(w_m(a))
    \end{equation}
    in place of $L(a)$. 

    Let $m\geq 2$, and let us denote by $ \mu_k(a; R_m)$ and $e_k(a;R_m)$ the continuations of the $k$-th eigenvalue and eigenvector of $L_m(0)$ along the $m$-th bifurcation branch. Recall, that, by Theorem~\ref{theorem_seq_bif} and Lemma \ref{lemma:diagonalization_L}, the second eigenvalue of $L_m(0)$ is negative for $m\geq 2$, that is, we now have $\mu_1(a=0;R_m)<0$. 

    We note first that, for every $m\geq 2$, Lemma \ref{lemma:lower_bound_rayleigh} holds for $L_m(a)$ too. To see this, we repeat all steps in its proof until the chain of inequalities leading to \eqref{eq:lower_bound_mu2_half}. Then, since we now have $\mu_0(0; R_m)<\mu_1(0;R_m)<\mu_2(0;R_m)\leq 0$, we simply replace the final lower bound $\frac{1}{2}\mu_2(0)$ in \eqref{eq:lower_bound_mu2_half} by $\mu_2(0;R_m)-\tilde{\varepsilon}$ for some sufficiently small $\tilde{\varepsilon}>0$ so that the inequality
    \begin{equation}
            \mu_1(a;R_m) < \mu_2(0;R_m)-\tilde{\varepsilon}
    \end{equation}
    holds for $\abs{a}$ small enough. From this, we deduce that Proposition \ref{proposition:infimum_attained} holds for $L_m(a)$ too. As a consequence, plugging equation \eqref{eq:one_dimensional} in the Rayleigh quotient for $L_m(0)$ we arrive to an expression analogous to \eqref{eq:R_1_before_taylor}; namely,
    \begin{equation}\label{eq:rayleigh_for_Lm}
         \mathcal{R}_m(a) := \inf_{v\in V_0(0)^\perp} \frac{\expval{v,L_m(a)v}}{\expval{v,v}} = \mu_0(a;R_m) \frac{\gamma Q_m(a)^2}{1+\gamma Q_m(a)^2} + \mu_1(a;R_m) \frac{1}{1+ \gamma Q_m(a)^2},
     \end{equation} 
     with $\gamma = \expval{e_0,e_0}/\expval{e_1,e_1}$ and
     \begin{equation}
         Q_m(a):= \frac{\int_{-\pi}^\pi e_1(a;R_m)(s)\dd{s}}{\int_{-\pi}^\pi e_0(a;R_m)(s)\dd{s}}.
     \end{equation}
     
     Recall that, since $m\geq 2$, we now have $\mu_0(0;R_m)<\mu_1(0;R_m)<0$, and, therefore, also $\mu_0(a;R_m)<\mu_1(a;R_m)<0$ for $\abs{a}$ small enough. Hence, from \eqref{eq:rayleigh_for_Lm} we deduce that $\mathcal{R}_m(a)<0$ for $\abs{a}$ small enough. Therefore, by Lemma \ref{lemma_cylindrical} we conclude that $E_m(a)$ is unstable for $\abs{a}$ small enough, depending only on $\alpha$, $\beta$, and $m$.
\end{proof}

\begin{proof}[Proof of Corollary \ref{cor::classification_minimizers}]
    We assume that $\alpha\in [\alpha_0,1)$; the proof for $\alpha\in (0,1)$ is essentially the same assuming $R\neq R_1$, with the obvious modifications.

First of all, we recall that, by Theorem~\ref{theorem_seq_bif}, there exists a neighbourhood of $\{(R, u_R):R>0\}$ in $\mathbb{R}\times C^{1,\beta}_{even}(\mathbb{T}^1)$ such that all the solutions of the equation $H_\alpha(u)=h_R$ in this neighbourhood are the trivial solutions $(R,u_R)$, with $R>0$, and the functions given by \eqref{eq::expression_solutions}. In other words, for every $R>0$ there exists $\varepsilon > 0$, depending only on $\alpha$, $R$, and $\beta$, such that if $(R', v)\in \mathbb{R}\times C^{1,\beta}_{even}(\mathbb{T}^1)$ solves $H_{\alpha}(v)=h_{R'}$, and $\abs{R'-R}+\norm{v-u_R}_{C^{1,\beta}_{even}(\mathbb{T}^1)}<\varepsilon$, then, either $v=u_{R'}$, or $R'=\gamma_m(a)R_m$ and $v=w_m(a)$ for some $m\geq 1$ and some $a\in (-\nu_m, \nu_m)$. 

    Now, for a given $R>0$, let $E_k = \{(z_1,z_2) : \abs{z_2} < u_k(z_1)\}\subset \mathbb{R}^2$ be a sequence of even and cylindrically symmetric volume-constrained critical points of $\mathcal{P}_\alpha$, each with constant nonlocal mean curvature equal to $h_k$, such that $u_k$ converges to $u_R\equiv R$ in $C^{1,\beta}_{even}(\mathbb{T}^1)$. Since the map $C^{1,\beta}_{even}(\mathbb{T}^1)\ni v \mapsto H_\alpha(v)$ is continuous, we have that $h_k \to h_R$ as $k\to +\infty$. Moreover, since, by Lemma \ref{lemma_H_Cinfty} and expression \eqref{eq:expression_hR}, the map $\tilde{R} \mapsto h_{\tilde{R}}$ is continuous and decreasing (and therefore bijective), we have that, for all $k$ large enough, $h_k = h_{R_k'}$ for some positive sequence $(R_k')$ converging to $R$ as $k\to +\infty$. As a consequence, for every $k$ large enough, $u_k$ solves the equation $H_\alpha(u_k)=h_{R_k'}$ for some $R_k'>0$, and $\smabs{R_k'-R}+\norm{u_k-u_R}_{C^{1,\beta}_{even}(\mathbb{T}^1)}< \varepsilon$. By our earlier observation, we deduce then that, for every $k$ large enough, either $u_k = u_{R_k'}$, or $R_k'=\gamma_m(a_k)R_m$ and $u_k=w_m(a_k)$ for some $m\geq 1$ and some $a_k\in (-\nu_m, \nu_m)$.

    Assume now that every $E_k$ is stable. Since, by Theorem \ref{theorem_stability_near_cylinders}, all non-straight near-cylinders and all straight cylinders of radius $R<R_1$ are unstable, we conclude that, for all $k$ large enough, $u_k = u_{R_k'}$, i.e., $E_k$ is a straight cylinder of radius $R_k'$ and, moreover, $R_k' \geq R_1$.
\end{proof}

\section{Computations for the Rayleigh quotient on the near-cylinders}\label{sec:computations}

In this section, we include all the computations needed to prove propositions \ref{proposition:dotmu_ddotmu_dote1} and \ref{proposition:asymptotic_alpha_to_1} that were omitted in the previous sections. 

\subsection{Proof of Proposition \ref{proposition:dotmu_ddotmu_dote1}}\label{subsec:explicit_rayleigh}

In this subsection we will prove Proposition \ref{proposition:dotmu_ddotmu_dote1}. To simplify the notation, following \cite{CNLMCDelCyl} and \cite{AlcoverBruera}, given a function $\varphi : \mathbb{R}\to \mathbb{R}$ and $s,t \in \mathbb{R}$, we will denote
\begin{equation}
    \delta_-\varphi(s,t) = \varphi(s)-\varphi(s-t)
\end{equation}
and
\begin{equation}
    \delta_0\varphi(s,t) = \varphi(s)+\varphi(s-t)
\end{equation}
throughout the whole section.

Similarly as we did in Subsection \ref{subsec:a_general_expression_for_ddot_mu}, in the following, all expressions are assumed to be evaluated at $a=0$ unless otherwise stated. In particular, we will write
\begin{equation}
    e_k(s) = e_k(a=0)(s) = \cos(ks).
\end{equation}

To prove Proposition \ref{proposition:dotmu_ddotmu_dote1} we need to evaluate expressions \eqref{eq:dot_mu}, \eqref{eq:dot_e1}, and \eqref{eq:ddot_mu} derived in Subsection \ref{subsec:a_general_expression_for_ddot_mu} for general families of symmetric operators in the case in which $L(a)$ is the linearized NMC operator on the first branch of near-cylinders. We repeat them here for easier reference:
\begin{equation}\label{eq:dot_mu_sec6}
     \dot{\mu}_1 = \frac{\expval{e_1,\dot{L}e_1}}{\expval{e_1,e_1}},
 \end{equation} 
 \begin{equation}\label{eq:dot_e1_sec6}
    \dot{e}_1 = \sum_{k\neq 1} \frac{1}{\mu_1 - \mu_k}\frac{\expval{e_k, \dot{L} e_1}}{\expval{e_k,e_k}}e_k,
\end{equation}
and
\begin{equation}\label{eq:ddot_mu_sec6}
    \ddot{\mu}_1 = \frac{\expval{e_1, \ddot{L}e_1}}{\expval{e_1,e_1}} +  2\sum_{k\neq 1} \frac{1}{\mu_1-\mu_k}\frac{\smabs{\expval{e_k,\dot{L}e_1}}^2}{\expval{e_k,e_k}\expval{e_1,e_1}}.
\end{equation}
Instead of \eqref{def_L(a)}, it will more convenient to use the following equivalent expression for the operator $L(a)=\frac{1}{2}H_{\alpha}(w_1(a))$,
\begin{equation}\label{def_L(a)_sec6}
    L(a)\psi  = \int_{\mathbb{R}} \delta_-\psi K(t,\delta_-w_1(a))\dd{t} - \int_{\mathbb{R}} \delta_0\psi K(t,\delta_0 w_1(a))\dd{t},
\end{equation}
where, recall,
\begin{equation}
    w_1(a)(s) = \gamma_1(a)R_1 + a(\cos(s) + v_1(a)(s)). 
\end{equation}

\begin{proof}[Proof of Proposition \ref{proposition:dotmu_ddotmu_dote1}]
We compute $\dot{L}(a)=\frac{1}{2}D_aD_{u}H_{\alpha}(w_1(a))$ first. Differentiating expression \eqref{def_L(a)_sec6} with respect to $a$ we obtain
\begin{equation}\label{dot_La_sec6}
    \dot{L}(a) \psi =\int_{\mathbb{R}} \left\{\delta_- \psi \delta_- \dot{w}_1(a) K_{z_2}\left(t, \delta_- w_1(a)\right) - \delta_0 \psi \delta_0 \dot{w}_1(a) K_{z_2}\left(t, \delta_0 w_1(a)\right)\right\}\dd{t}.
\end{equation}
Using that $\dot{w}_1(0) = e_1$ (since $\dot{\gamma}_1(0)=0$ by Proposition \ref{proposition:dot_gamma}) and $K_{z_2}(t,0)=0$, we obtain, at $a=0$,
\begin{align}
   \dot{L} e_1 &= \int_{\mathbb{R}} -\left(\delta_0 e_1\right)^2 K_{z_2}(t, 2R_1)\dd{t}\\ 
   &= -\frac{1}{2}\int_{\mathbb{R}} 
       \left\{2[1+\cos(t)]e_0+[1+2\cos(t)+\cos(2t)]e_{2}\right\}K_{z_2}(t,2R_1)\dd{t},
  \label{L'_en}
\end{align}
where we have used the identities
\begin{align}
    \left(\delta_0 e_1\right)^2(s,t)  &= (\cos(s)+\cos(s-t))^2 \\ 
    &= ( (1+\cos(t))\cos(s)+\sin(t)\sin(s))^2 \\ 
    &= \begin{multlined}[t][0.65\displaywidth]
        (1+\cos(t))^2\cos^2(s) +\sin^2(t)\sin^2(s)\\ +2(1+\cos(t))\sin(t)\cos(s)\sin(s)
    \end{multlined}\\
    &= \begin{multlined}[t][0.65\displaywidth]
        (1+\cos(t))^2\frac{1}{2}(1+\cos(2s)) +\sin^2(t)\frac{1}{2}(1-\cos(2s))\\ +2(1+\cos(t))\sin(t)\cos(s)\sin(s)
    \end{multlined}\\
    &= \begin{multlined}[t][0.65\displaywidth]
        \frac{1}{2}\left\{\left[(1+\cos(t))^2+\sin^2(t)\right]e_0(s) +\left[(1+\cos(t))^2-\sin^2(t)\right]e_{2}(s)\right\}\\ +2(1+\cos(t))\sin(t)\cos(s)\sin(s)
    \end{multlined}\\
    &= \begin{multlined}[t][0.65\displaywidth]
        \frac{1}{2}\left\{2\left[1+\cos(t)\right]e_0(s) +\left[1+2\cos(t)+\cos(2t)\right]e_{2}(s)\right\}\\ +2(1+\cos(t))\sin(t)\cos(s)\sin(s),
    \end{multlined}
\end{align}
and the fact that the integral of the last term multiplied by $K_{z_2}(t,2R_1)$ vanishes because it is odd in $t$, and $K_{z_2}(t,2R_1)$ is even. Hence, 
\begin{equation}
    \dot{\mu}_{1} = \frac{\expval{e_1,\dot{L}e_1}}{\expval{e_1,e_1}} = 0.
\end{equation}

From the calculations above, we also deduce that
\begin{equation}\label{eq:mat_elements_dotL}
    \frac{\expval{e_k,\dot{L}e_1}}{\expval{e_k,e_k}} = \begin{cases}\displaystyle
        -\int_{\mathbb{R}} (1+\cos(t))K_{z_2}(t,2R_1)\dd{t}, &\text{ if }k=0,\\[10 pt]
        \displaystyle -\frac{1}{2}\int_{\mathbb{R}}[1+2\cos(t)+\cos(2t)]K_{z_2}(t,2R_1)\dd{t}, &\text{ if }k=2,\\[10 pt]
        \displaystyle 0, &\text{ otherwise.}
    \end{cases}
\end{equation}
As a consequence, recalling \eqref{eq:dot_e1_sec6}, \eqref{eq:ddot_mu_sec6}, and the fact that $\mu_1(0)=0$ (since $R=R_1$), we see that
\begin{equation}
    \dot{e}_1 = -\frac{1}{\mu_0}\frac{\expval{e_0,\dot{L}e_1}}{\expval{e_0,e_0}\expval{e_1,e_1}}e_0 - \frac{1}{\mu_2}\frac{\expval{e_2,\dot{L}e_1}}{\expval{e_2,e_2}\expval{e_1,e_1}}e_2
\end{equation}
and
\begin{equation}
    \ddot{\mu}_1=\frac{\expval{e_1, \ddot{L}e_1}}{\expval{e_1,e_1}} - \frac{2}{\mu_0}\frac{\smabs{\expval{e_0,\dot{L}e_1}}^2}{\expval{e_0,e_0}\expval{e_1,e_1}}- \frac{2}{\mu_2}\frac{\smabs{\expval{e_2,\dot{L}e_1}}^2}{\expval{e_2,e_2}\expval{e_1,e_1}},
\end{equation}
as we claimed in Proposition \ref{proposition:dotmu_ddotmu_dote1}.
\end{proof}

\subsection{Asymptotic behaviour as $\alpha \uparrow 1$: proof of Proposition \ref{proposition:asymptotic_alpha_to_1}}\label{subsec:asymptotic_alpha_1}
In this subsection we will prove Proposition \ref{proposition:asymptotic_alpha_to_1}. To do this, first we must explicitly evaluate formula \eqref{eq:asymptotic_expression_ddot_R1} for $\ddot{\mathcal{R}}_1(0)$, which we re-state next for easier reference,
\begin{equation}\label{eq:asymptotic_expression_ddot_R1_sec6}
        \mathcal{R}_1(a) = \left(\frac{1}{2}\frac{\expval{e_1, \ddot{L}e_1}}{\expval{e_1,e_1}} - \frac{1}{\mu_{2}}\frac{\smabs{\expval{e_{2},\dot{L}e_1}}^2}{\expval{e_{2},e_{2}}\expval{e_1,e_1}}\right)\Bigg\vert_{a=0} a^2 + O(a^3).
\end{equation}

Notice that the second term in \eqref{eq:asymptotic_expression_ddot_R1_sec6} was already obtained in \eqref{eq:mat_elements_dotL}. It remains to compute $\ddot{L}(a)$. Differentiating \eqref{dot_La_sec6} with respect to $a$ we get
\begin{multline}
    \ddot{L}(a)\psi = \int_{\mathbb{R}} \left\{\delta_- \psi \delta_- \ddot{w}_1(a) K_{z_2}\left(t, \delta_- w_1(a)\right) - \delta_0 \psi \delta_0 \ddot{w}_1(a) K_{z_2}\left(t, \delta_0 w_1(a)\right)\right\}\dd{t}\\
    +\int_{\mathbb{R}} \left\{\delta_- \psi \left(\delta_- \dot{w}_1(a)\right)^2K_{z_2,z_2}\left(t, \delta_- w_1(a)\right) - \delta_0 \psi \left(\delta_0 \dot{w}_1(a)\right)^2 K_{z_2,z_2}\left(t, \delta_0 w_1(a)\right)\right \}\dd{t},
\end{multline}
which, since $\dot{w}_1(0)=e_1$, $\ddot{w}_1(0) = \ddot{\gamma}_1(0) R_1 + 2 \dot{v}_1(0)$, and $K_{z_2}(t,0)=0$, at $a=0$ yields
\begin{multline}\label{ddotL}
    \ddot{L} e_1 =  \int_{\mathbb{R}} - \delta_0 e_1 \left(2\ddot{\gamma}_1 R_1   + 2\delta_0 \dot{v}_1\right) K_{z_2}(t, 2R_1)\dd{t}\\
    +\int_{\mathbb{R}} \left\{ (\delta_- e_1)^3K_{z_2,z_2}(t, 0) - \left(\delta_0 e_1\right)^3 K_{z_2,z_2}(t, 2R_1)\right\}\dd{t}.
\end{multline}

Notice that calculating $\expval{e_1, \ddot{L}e_1}$ requires knowing $\ddot{\gamma}_1(0)$ and $\dot{v}_1(0)$. In order to do this, we proceed as follows. Slightly abusing notation, we write
\begin{equation}
    \overline{\Phi}_1(a)=\overline{\Phi}_1(a, \gamma_1(a), v_1(a)) = \frac{1}{2a}\left\{H_{\alpha}(\gamma_1(a) R_1 + a\left(e_1 + v_1(a)\right))-H_{\alpha}(\gamma_1(a) R_1)\right\},
\end{equation}
as in the proof of Theorem \ref{theorem_seq_bif}. Since $\overline{\Phi}_1(a)\equiv 0$, we have 
\begin{equation}\label{first_total_derivative}
    0=\derivative{\overline{\Phi}_1}{a} = \partial_a \overline{\Phi}_1+\dot{\gamma}_1\partial_\gamma {\overline{\Phi}_1}+\partial_v {\overline{\Phi}_1}[\dot{v}_1],
\end{equation}
which, since $\dot{\gamma}_1(0)=0$ (by Proposition \ref{proposition:dot_gamma}), yields
\begin{equation}\label{first_tot_der_at_0}
    0=\partial_a \overline{\Phi}_1 + \partial_v {\overline{\Phi}_1}[\dot{v}_1].
\end{equation}
Differentiating \eqref{first_total_derivative} again we obtain
\begin{equation}
\setlength\arraycolsep{1.5pt}
    \begin{array}{rcccccccccc}
    0=\displaystyle\dv[2]{\overline{\Phi}_1}{a}  & = & \partial^2_{a,a} \overline{\Phi}_1 &+& \dot{\gamma}_1\partial^2_{a, \gamma} \overline{\Phi}_1 &+& \partial^2_{a,v}\overline{\Phi}_1[\dot{v}_1] &&&&\\
    &+& \dot{\gamma}_1 \partial^2_{\gamma,a}\overline{\Phi}_1  &+& \dot{\gamma}_1^2 \partial^2_{\gamma,\gamma} \overline{\Phi}_1 &+& \dot{\gamma}_1\partial^2_{\gamma,v} \overline{\Phi}_1[\dot{v}_1] &&&&\\[4 pt]
    &+& \partial^2_{v, a} \overline{\Phi}_1 [\dot{v}_1]&+&\dot{\gamma}_1\partial^2_{v, \gamma}\overline{\Phi}_1[\dot{v}_1] &+& \partial^2_{v,v} \overline{\Phi}_1 [\dot{v}_1,\dot{v}_1]  &+& \ddot{\gamma}_1\partial_{\gamma}\overline{\Phi}_1 &+& \partial_{v} \overline{\Phi}_1 [\ddot{v}_1],
\end{array}
\end{equation}
which, recalling that $\dot{\gamma}_1(0)=0$, yields
\begin{equation}\label{second_total_derivative_at_a0}
    0 = \partial^2_{a,a} \overline{\Phi}_1  + 2\partial_{a,v}\overline{\Phi}_1[\dot{v}_1] + \partial_{v,v}^2 \overline{\Phi}_1 [\dot{v}_1,\dot{v}_1]+ \ddot{\gamma}_1\partial_{\gamma}\overline{\Phi}_1 + \partial_{v} \overline{\Phi}_1 [\ddot{v}_1].
\end{equation}

Since $\overline{\Phi}_1(a=0)=\overline{\Phi}_1(0,1,v)$ is linear in $v$, we have $\partial_{v,v}^2 \overline{\Phi}_1\vert_{a=0}\equiv 0$. Moreover, since $v_1(a)$ is orthogonal to $e_1(0)$ for all $a\in (-\nu_1,\nu_1)$, and the $e_k$ are all eigenvectors of $L(0)=\partial_v \overline{\Phi}_1 (0,1,0)$, we have that $\expval{e_1, \partial_{v} \overline{\Phi}_1 [\ddot{v}_1]} = \expval{e_1, L(0)\ddot{v}_1}= 0$. Hence, from \eqref{second_total_derivative_at_a0} and these two observations we deduce that
\begin{equation}\label{en_component_eq_0}
    0  = {\langle e_1, \partial^2_{a,a} \overline{\Phi}_1+2\partial_{a,v}\overline{\Phi}_1[\dot{v}_1]+\ddot{\gamma}_1\partial_{\gamma}\overline{\Phi}_1 \rangle}.
\end{equation}

We compute $\partial_{\gamma} \overline{\Phi}_1$ first. We have
\begin{equation}\label{eq:dgamma_phi}
    \partial_{\gamma}\overline{\Phi}_1(0,1,0) = -\int_{\mathbb{R}}\delta_0 e_1 2R_1K_{z_2}(t,2R_1)\dd{t}.
\end{equation}

We now need to compute $\partial_{a,v}\overline{\Phi}_1$ and $\partial_{a,a}^2 \overline{\Phi}_1$. We calculate $\partial_{a,v}\overline{\Phi}_1$ first. We have
\begin{equation}
    \partial_{v} \overline{\Phi}_1(a,1,0) [\psi] = \int_{\mathbb{R}} \left\{ \delta_-\psi K(t,a\delta_- e_1)  - \delta_0\psi K(t,2 \gamma_1 R_1+a\delta_0 e_1 )\right\}\dd{t},
\end{equation}
and thus
\begin{align}
    \partial_{a,v}\overline{\Phi}_1(0,1,0)[\psi] &=
    \int_{\mathbb{R}} \left\{\delta_-\psi \delta_-e_1 K_{z_2}(t,0) - \delta_0\psi \delta_0 e_1 K_{z_2}(t,2R_1)\right\} \dd{t}\\ &=- \int_{\mathbb{R}} \delta_0\psi \delta_0 e_1K_{z_2}(t,2R_1)\dd{t},\label{eq:dav_phi}
\end{align}
since $K_{z_2}(t,0)=0$.

Now we compute $\partial_{a,a}^2 \overline{\Phi}_1$. Recalling that $\overline{\Phi}_1(a, 1, 0) = \frac{1}{2a}\left\{H_{\alpha}(R_1 + ae_1)-H_{\alpha}(R_1)\right\}$, we have
\begin{equation}
    \partial_a \overline{\Phi}_1(a,1,0) = \frac{1}{2a}\left\{ D_uH_\alpha(R_1+a e_1)e_1-\overline{\Phi}_1(a, 1, 0)\right\},
\end{equation}
and thus,
\begin{equation}
    \partial_{a,a}^2 \overline{\Phi}_1(a, 1, 0) = \frac{1}{2a}\left\{D_{u,u}^2 H_{\alpha}(R_1+ a e_1)[e_1,e_1] - \frac{2}{a} \left\{D_uH_{\alpha}(R_1+ae_1)e_1 - \Phi_1(a,1,0)\right\}\right\}.
\end{equation}
Letting $a\to 0$ and using Lemma \ref{lemma_derivatives_calc1} (which we state and prove later on in this section) applied to the function $a\mapsto H_{\alpha} (R_1 + a e_1)$, we obtain
\begin{equation}\label{eq:daa_phi}
    \partial_{a,a}^2\overline{\Phi}_1 = \frac{1}{6}D_{u,u,u}^3 H_{\alpha}(R_1)[e_1,e_1,e_1] = \frac{1}{3} \int_{\mathbb{R}} \left\{(\delta_- e_1)^3K_{z_2,z_2}(t,0) -(\delta_0 e_1)^3 K_{z_2,z_2}(t,2R_1) \right\} \dd{t}.
\end{equation}

Substituting \eqref{eq:dgamma_phi}, \eqref{eq:dav_phi}, and \eqref{eq:daa_phi} in \eqref{en_component_eq_0} we get
\begin{multline}
    \left\langle e_1, \int_{\mathbb{R}}-\delta_0 e_1 (2\ddot{\gamma}_1 R_1 + 2\delta_0\dot{v}_1) K_{z_2}(t,2R_1)\dd{t} \right\rangle= \\
    -\frac{1}{3} \left\langle e_1, \int_{\mathbb{R}} \left\{(\delta_- e_1)^3K_{z_2,z_2}(t,0) -(\delta_0 e_1)^3 K_{z_2,z_2}(t,2R_1) \right\} \dd{t}\right \rangle,
\end{multline}
which, together with \eqref{ddotL}, yields
\begin{equation}\label{emddotLem}
    \expval{e_1,\ddot{L} e_1} = \left\langle e_1, \frac{2}{3}\int_{\mathbb{R}} \left\{ (\delta_- e_1)^3K_{z_2,z_2}(t, 0) - \left(\delta_0 e_1\right)^3 K_{z_2,z_2}(t, 2R_1)\right\}\dd{t}\right\rangle.
\end{equation}

Now, to compute the expression in the right hand side of \eqref{emddotLem}, we will use several trigonometric identities to expand $(\delta_- e_1)^3$ and $(\delta_0 e_1)^3$ into products of a function of $t$ and a function of $s$. We have
\begin{align}
    \left(\delta_- e_1\right)^3(s,t)  &= (\cos(s)-\cos(s-t))^3 \\ &= ( (1-\cos(t))\cos(s)-\sin(t)\sin(s))^3 \\ 
    &= \begin{multlined}[t][0.65\displaywidth]
        (1-\cos(t))^3\cos^3(s) +3(1-\cos(t))\sin^2(t)\cos(s)\sin^2(s)\\ -3(1-\cos(t))^2\sin(t)\cos^2(s)\sin(s)-\sin^3(t)\sin^3(s),
    \end{multlined}\\ 
    &= \begin{multlined}[t][0.65\displaywidth]
        (1-\cos(t))^3\frac{1}{4}\left(3e_1(s)+e_{3}(s)\right) \\+3(1-\cos(t))\sin^2(t)\frac{1}{4}\left(e_1(s)-e_{3}(s)\right)\\ -3(1-\cos(t))^2\sin(t)\cos^2(s)\sin(s)-\sin^3(t)\sin^3(s),
    \end{multlined}\\ 
    &= \begin{multlined}[t][0.65\displaywidth]
        \frac{3}{4}\left\{ (1-\cos(t))^3 + (1-\cos(t))\sin^2(t) \right\} e_{1}(s) \\
        +\frac{1}{4} \left\{ (1-\cos(t))^3 - 3(1-\cos(t))\sin^2(t)\right\} e_{3}(s)\\ -3(1-\cos(t))^2\sin(t)\cos^2(s)\sin(s)-\sin^3(t)\sin^3(s),
    \end{multlined}\\
    &= \begin{multlined}[t][0.65\displaywidth]
        \frac{3}{4}\left\{ 3-4\cos(t)+\cos(2t) \right\} e_{1}(s) \\
        +\frac{1}{4} \left\{ (1-\cos(t))^3 - 3(1-\cos(t))\sin^2(t)\right\} e_{3}(s)\\ -3(1-\cos(t))^2\sin(t)\cos^2(s)\sin(s)-\sin^3(t)\sin^3(s),
    \end{multlined}
\end{align}
and, similarly,
\begin{align}
    \left(\delta_0 e_1\right)^3(s,t)  &= (\cos(s)+\cos(s-t))^3 \\ &= ( (1+\cos(t))\cos(s)+\sin(t)\sin(s))^3 \\ 
    &= \begin{multlined}[t][0.65\displaywidth]
        (1+\cos(t))^3\cos^3(s) +3(1+\cos(t))\sin^2(t)\cos(s)\sin^2(s)\\ +3(1+\cos(t))^2\sin(t)\cos^2(s)\sin(s)+\sin^3(t)\sin^3(s),
    \end{multlined}\\ 
    &= \begin{multlined}[t][0.65\displaywidth]
        (1+\cos(t))^3\frac{1}{4}\left(3e_1(s)+e_{3}(s)\right) \\+3(1+\cos(t))\sin^2(t)\frac{1}{4}\left(e_1(s)-e_{3}(s)\right)\\ +3(1+\cos(t))^2\sin(t)\cos^2(s)\sin(s)+\sin^3(t)\sin^3(s),
    \end{multlined}\\ 
    &= \begin{multlined}[t][0.65\displaywidth]
        \frac{3}{4}\left\{ (1+\cos(t))^3 + (1+\cos(t))\sin^2(t) \right\} e_{1}(s) \\
        +\frac{1}{4} \left\{ (1+\cos(t))^3 - 3(1+\cos(t))\sin^2(t)\right\} e_{3}(s)\\ +3(1+\cos(t))^2\sin(t)\cos^2(s)\sin(s)+\sin^3(t)\sin^3(s),
    \end{multlined}\\
    &= \begin{multlined}[t][0.65\displaywidth]
        \frac{3}{4}\left\{ 3+4\cos(t)+\cos(2t) \right\} e_{1}(s) \\
        +\frac{1}{4} \left\{ (1-\cos(t))^3 - 3(1-\cos(t))\sin^2(t)\right\} e_{3}(s)\\ -3(1-\cos(t))^2\sin(t)\cos^2(s)\sin(s)-\sin^3(t)\sin^3(s).
    \end{multlined}
\end{align}
Notice that the integral of the last two terms in the expressions for $(\delta_- e_1)^3$ and $(\delta_0 e_1)^3$ against $K_{z_2,z_2}(t,0)$ and $K_{z_2,z_2}(t,2R_1)$, respectively, vanish (in the principal value sense) because they are odd in $t$, and $K_{z_2,z_2}(t,0)$ and $K_{z_2,z_2}(t,2R_1)$ are even in $t$.

As a consequence, from \eqref{emddotLem} we obtain
\begin{multline}\label{eq:exp_val_ddotL}
    \frac{\expval{e_1,\ddot{L} e_1}}{\expval{e_1,e_1}} = \frac{1}{2} \int_{\mathbb{R}} \left(3-4\cos(t)+\cos(2t)\right) K_{z_2,z_2}(t,0) \dd{t} \\ -\frac{1}{2}\int_{\mathbb{R}} \left(3+4\cos(t)+\cos(2t)\right)K_{z_2,z_2}(t,2R_1) \dd{t}.
\end{multline}

Now, since by Lemma \ref{lemma:diagonalization_L},
\begin{equation}
    \mu_{2} = \int_{\mathbb{R}}\left\{(1-\cos(2t))K(t,0)-(1+\cos(2t))K(t,2R_1)\right\}\dd{t},
\end{equation}
we have, by \eqref{eq:mat_elements_dotL},
\begin{equation}
   \frac{-1}{\mu_{2}}\frac{\smabs{\expval{e_{2}, \dot{L} e_1}}^2}{\expval{e_{2},e_{2}}\expval{e_1,e_1}} = -\frac{1}{4}
   \frac{\left(\int_{\mathbb{R}}[1+2\cos(t)+\cos(2t)]K_{z_2}(t,2R_1)\dd{t}\right)^2}{\int_{\mathbb{R}}\left\{(1-\cos(2t))K(t,0)-(1+\cos(2t))K(t,2R_1)\right\}\dd{t}}.
\end{equation}
Plugging this and \eqref{eq:exp_val_ddotL} in \eqref{eq:asymptotic_expression_ddot_R1_sec6} we get
\begin{multline}\label{final_expression_ddot_mu}
    \ddot{\mathcal{R}}_1(0) =
    \frac{1}{4} \int_{\mathbb{R}} \left(3-4\cos(t)+\cos(2t)\right) K_{z_2,z_2}(t,0) \dd{t} \\ 
    -\frac{1}{4} \int_{\mathbb{R}}  \left(3+4\cos(t)+\cos(2t)\right)K_{z_2,z_2}(t,2R_1) \dd{t} \\ 
    -\frac{1}{4}
   \frac{\left(\int_{\mathbb{R}}[1+2\cos(t)+\cos(2t)]K_{z_2}(t,2R_1)\dd{t}\right)^2}{\int_{\mathbb{R}}\left\{(1-\cos(2t))K(t,0)-(1+\cos(2t))K(t,2R_1)\right\}\dd{t}}.
\end{multline}

We now state and prove Lemma \ref{lemma_derivatives_calc1} that we used in our previous calculations. Its proof is a simple calculus exercise, but we include it for completeness.
\begin{lemma}\label{lemma_derivatives_calc1}
Let $f: \mathbb{R}\to \mathbb{R}$ be a function three times differentiable. Then,
\begin{equation}
    \lim_{\delta\to 0}\frac{f''(\delta)-2\frac{f'(\delta)-\frac{1}{\delta}[f(\delta)-f(0)]}{\delta}}{\delta}  = \frac{1}{3}f'''(0).
\end{equation}
\end{lemma}
\begin{proof}
It is a simple calculus exercise. We write
\begin{equation}
    f(\delta)=f(0)+f'(0)\delta+\frac{1}{2}f''(0)\delta^2 + \frac{1}{3!}f'''(0)\delta^3 + o(\delta^3)
\end{equation}
and similarly for $f'(\delta)$ and $f''(\delta)$. Hence,
\begin{align}
    \frac{f'(\delta)-\frac{1}{\delta}[f(\delta)-f(0)]}{\delta} &= \frac{f'(\delta)-f'(0)-\frac{1}{2}f''(0)\delta-\frac{1}{3!}f'''(0)\delta^2 + o(\delta^2)}{\delta}\\ 
    &= \frac{f''(0)\delta + \frac{1}{2}f'''(0) \delta^2-\frac{1}{2}f''(0)\delta-\frac{1}{3!}f'''(0)\delta^2 + o(\delta^2)}{\delta} \\ 
    &= \frac{1}{2}f''(0) + \left(\frac{1}{2}-\frac{1}{3!}\right) f'''(0)\delta + o( \delta)
\end{align}
and thus
\begin{align}
    \frac{f''(\delta)-2\frac{f'(\delta)-\frac{1}{\delta}[f(\delta)-f(0)]}{\delta}}{\delta} &= \frac{f''(\delta)-f''(0)-\left(1-\frac{1}{3}\right)f'''(0)\delta+o(\delta)}{\delta}\\ 
    &=  \frac{f'''(0)\delta-\left(1-\frac{1}{3}\right)f'''(0)\delta+o(\delta)}{\delta} \\ 
    &= \frac{1}{3}f'''(0)+o(1).
\end{align}
\end{proof}

\begin{figure}[h!]
    \centering
    \includegraphics[trim = {3cm 9.5cm 3cm 9.8cm}, clip, width = 0.6\textwidth]{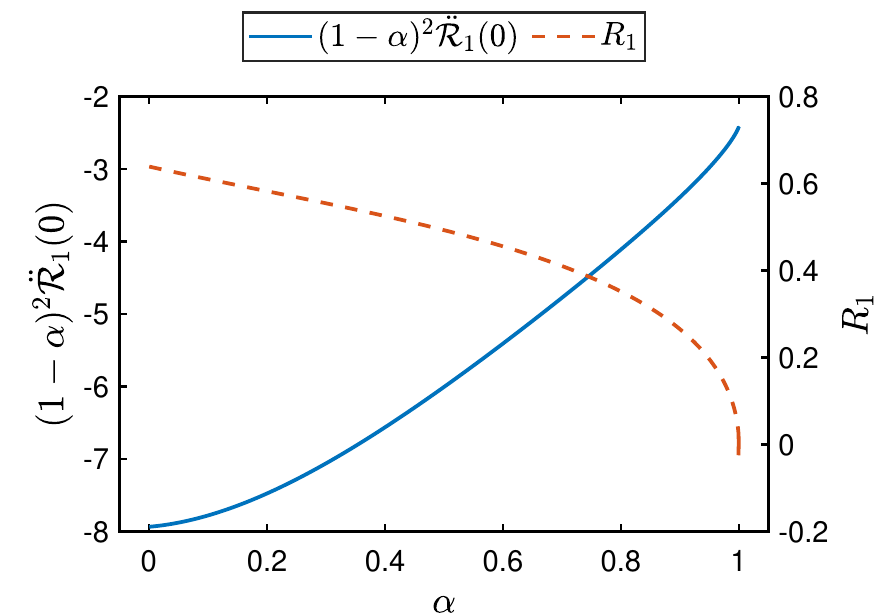}
    \caption{Representation of the values of $R_1$ and $\ddot{\mathcal{R}}_1$ as a function of $\alpha$ obtained by numerical computation of the function in \eqref{expression_ddot_mu_hom_kernel}. Since $\ddot{\mathcal{R}}_1(0)$ diverges quadratically as $\alpha\uparrow 1$, we represent its product with $(1-\alpha)^2$. Notice how $\ddot{\mathcal{R}}_1(0)<0$ for all $\alpha\in (0,1)$.}\label{fig:stability_vs_alpha}
\end{figure} 

Finally, we are ready to give the proof of Proposition \ref{proposition:asymptotic_alpha_to_1}.
\begin{proof}[Proof of Proposition \ref{proposition:asymptotic_alpha_to_1}]
We will now compute the integrals in \eqref{final_expression_ddot_mu} explicitly, and we will evaluate them in the limit $\alpha \uparrow 1$. Since $K(z)=\abs{z}^{-2- \alpha}$, we have
\begin{equation}
K_{z_2}(z_1,z_2)=-\frac{(2+\alpha)z_2}{(z_1^2+z_2^2)^{(4+\alpha)/2}} \quad \text{and}\quad K_{z_2,z_2}(z_1,z_2) = -\frac{2+\alpha}{(z_1^2+z_2^2)^{(4+\alpha)/2}}+\frac{(2+\alpha)(4+\alpha)z_2^2}{(z_1^2+z_2^2)^{(6+\alpha)/2}}.
\end{equation}
Substituting these expressions in \eqref{final_expression_ddot_mu} we deduce that
\begin{multline}\label{eq:expression_ddotR_for_explicit_kernel}
    \frac{4}{2+\alpha}\ddot{\mathcal{R}}_1(0) = -\int_{\mathbb{R}}  \frac{3-4\cos(s)+\cos(2s)}{\abs{s}^{4+\alpha}}\dd{s}  \\ + \int_{\mathbb{R}} (3+4\cos(s)+\cos(2s))\left\{\frac{1}{(s^2+c^2)^{(4+\alpha)/2}}-\frac{(4+\alpha)c^2 }{(s^2+c^2)^{(6+\alpha)/2}}\right\} \dd{s}\\ 
    - \frac{(2+\alpha) c^2 \left(\int_{\mathbb{R}} \frac{1+2\cos( s)+\cos(2s)}{(s^2+c^2)^{(4+\alpha)/2}}\dd{s} \right)^2}{\int_{\mathbb{R}}\left\{\frac{1-\cos(2s)}{\abs{s}^{2+\alpha}}-\frac{1+\cos(2s)}{(s^2+c^2)^{(2+\alpha)/2}}\right\}\dd{s}},
\end{multline}
where we have denoted $c:=2R_1$. 

We will first compute the integrals
\begin{equation}
    I_{4+\alpha} := \int_{\mathbb{R}}  \frac{3-4\cos(s)+\cos(2s)}{\abs{s}^{4+\alpha}}\dd{s}= \int_{\mathbb{R}} \frac{8\sin^4(s/2)}{\abs{s}^{4+\alpha}} \dd{s}
\end{equation}
and
\begin{equation}
    J_{2+\alpha} := \int_\mathbb{R}\frac{1-\cos(2s)}{\abs{s}^{2+\alpha}}\dd{s} = \int_{\mathbb{R}}\frac{2\sin^2(s)}{\abs{s}^{2+\alpha}}\dd{s}.
\end{equation}
Using the following formula, which can be easily proved using integration by parts,
\begin{equation}
    \int_0^\infty \frac{\sin^k(s)}{s^{k+\alpha}}\dd{s} = \frac{1}{(k+\alpha-1)(k+\alpha-2)\cdots \alpha} \int_0^\infty \derivative[k]{\sin^k(s)}{s} \frac{1}{s^\alpha}\dd{s},
\end{equation}
and the fact (see, for example, \cite[3.761, 9]{tables_book} and apply Euler's reflection and Legendre's duplication formulae) that
\begin{equation}
    \int_{\mathbb{R}}\frac{\cos(\xi s)}{\abs{s}^{\alpha}} \dd{s} = \int_{\mathbb{R}}\frac{e^{i\xi s}}{\abs{s}^{\alpha}} \dd{s} = \frac{2^{1-\alpha}\sqrt{\pi}\Gamma((1-\alpha)/2)}{\Gamma(\alpha/2)}\abs{\xi}^{\alpha-1}
\end{equation}
for $\alpha\in (0,1)$, we deduce that
\begin{align}
    I_{4+\alpha} = 16\cdot 2^{-3-\alpha} \int_0^\infty \frac{\sin^4(s)}{s^{4+\alpha}}\dd{s} &=  \frac{16\cdot 2^{-3 - \alpha}}{(3+\alpha)(2+\alpha)(1+\alpha) \alpha } \int_0^\infty \frac{32 \cos(4s)-8\cos(2s)}{s^{\alpha}}\dd{s}\\ &=  \frac{2^{3-\alpha}\left(2^{\alpha+1}-1\right)\sqrt{\pi}}{(3+\alpha)(2+\alpha)(1+\alpha) \alpha }\frac{\Gamma\left((1-\alpha)/2\right) }{\Gamma(\alpha/2)}\\ 
    &= \frac{\left(8-2^{2-\alpha}\right)\sqrt{\pi}}{(3+\alpha)(2+\alpha)(1+\alpha)}\frac{\Gamma\left((1-\alpha)/2\right) }{\Gamma((2+\alpha)/2)}
\end{align}
and
\begin{equation}
    J_{2+\alpha} = \frac{4\sqrt{\pi}  }{(1+\alpha)\alpha }\frac{\Gamma((1-\alpha)/2)}{\Gamma(\alpha/2)} = \frac{2\sqrt{\pi}  }{1+\alpha }\frac{\Gamma((1-\alpha)/2)}{\Gamma((2+\alpha)/2)}.
\end{equation}

The other integrals in \eqref{eq:expression_ddotR_for_explicit_kernel} are computed using the relation (see \cite[3.771, 2]{tables_book} and apply Euler's reflection formula)
\begin{equation}
    F_{\nu}(\xi) := \int_{\mathbb{R}} \frac{\cos(\xi s)}{(s^2+c^2)^{\nu/2}} \dd{s} = \int_{\mathbb{R}} \frac{e^{i \xi s}}{(s^2+c^2)^{\nu/2}} \dd{s}= \begin{cases}
        \frac{\sqrt{\pi} }{c^{\nu-1}\Gamma(\nu/2)}\Gamma\left( \frac{\nu-1}{2}\right), &\text{ if }\xi=0\\[6 pt]
        \frac{\sqrt{\pi} }{c^{\nu-1}\Gamma(\nu/2)} \Psi_{\frac{\nu-1}{2}}(c \xi), &\text{ if }\xi>0
    \end{cases}
\end{equation}
for $\nu>1$, where
\begin{equation}
    \Psi_{\frac{\nu-1}{2}}(z):=2\left(\frac{z}{2}\right)^{(\nu-1)/2} K_{\frac{\nu-1}{2}}(z),
\end{equation}
and $K_{\nu}$ denotes the modified Bessel function of the second kind.

Hence, we can write
\begin{multline}
    \frac{4}{2+\alpha}\ddot{\mathcal{R}}_1(0) = -I_{4+\alpha} + 3F_{4+\alpha}(0)+4F_{4+\alpha}(1)+F_{4+\alpha}(2)\\
    -(4+\alpha)c^2 \left(3F_{6+\alpha}(0)+4F_{6+\alpha}(1)+F_{6+\alpha}(2)\right)\\ 
    - (2+\alpha) c^2 \frac{\left(F_{4+\alpha}(0)+2F_{4+\alpha}(1)+F_{4+\alpha}(2)\right)^2}{J_{2+\alpha}-F_{2+\alpha}(0)-F_{2+\alpha}(2)},
\end{multline}
or, equivalently,
\begin{multline}\label{expression_ddot_mu_hom_kernel}
\frac{\Gamma\left(\frac{2+\alpha}{2}\right)}{\sqrt{\pi}/2} c^{3+ \alpha} \ddot{\mathcal{R}}_1(0)= 
-\frac{\left(4-2^{1-\alpha}\right)c^{3+\alpha}}{(3+\alpha)(1+\alpha)}\Gamma\left(\frac{1-\alpha}{2}\right) \\
+ \left( 3\Gamma\left(\frac{3+\alpha}{2}\right)+4 \Psi_{\frac{3+\alpha}{2}}(c)+\Psi_{\frac{3+\alpha}{2}}(2c)\right)\\ 
- 2\left(3\Gamma\left(\frac{5+\alpha}{2}\right)+4 \Psi_{\frac{5+\alpha}{2}}(c)+\Psi_{\frac{5+\alpha}{2}}(2c)\right)\\
-2\frac{\left( \Gamma\left(\frac{3+\alpha}{2}\right)+2\Psi_{\frac{3+\alpha}{2}}(c)+\Psi_{\frac{3+\alpha}{2}}(2c)\right)^2}{\frac{2}{(1+\alpha)}c^{1+\alpha}\Gamma\left(\frac{1-\alpha}{2}\right)-\Gamma\left(\frac{1+\alpha}{2}\right)-\Psi_{\frac{1+\alpha}{2}}(2c)}.
\end{multline}

Let us now analyse the limit $\alpha\uparrow 1$. The equation for $c=2R_1$ as a function of $\alpha$, namely,
\begin{equation}
    \mu_1=\int_{\mathbb{R}} (1-\cos(t)) K(t,0) \dd{t} - \int_{\mathbb{R}}(1+\cos(t))K(t,c)\dd{t} = 0,
\end{equation}
when $K(z)=\abs{z}^{-2-\alpha}$ is
\begin{equation}\label{eq_defining_c}
    \int_{\mathbb{R}} \frac{2\sin^2(\frac{t}{2})}{\abs{t}^{2+\alpha}} \dd{t} =  \int_{\mathbb{R}}\frac{1+\cos(t)}{(t^2+c^2)^{\frac{2+\alpha}{2}}}\dd{t}.
\end{equation}
Since the expression in the left hand side diverges as $\alpha\uparrow 1$, it follows that $c \to 0$ as $\alpha\uparrow 1$. On the other hand, using this, and changing variables in the definition of $F_\nu$, it is easy to see that
\begin{equation}
    \Psi_{\frac{3+\alpha}{2}}(c(\alpha)) \to \Gamma\left(\frac{3+1}{2}\right) \quad\text{as}\quad \alpha \uparrow 1,
\end{equation}
and similarly for similar terms. 

Now, let us calculate the asymptotic behaviour of $c$ as $\alpha\uparrow 1$. Repeating similar computations as before, we find that
\begin{equation}
    \int_{\mathbb{R}} \frac{2\sin^2(\frac{t}{2})}{\abs{t}^{2+\alpha}} \dd{t} = 2^{-\alpha}\int_{\mathbb{R}} \frac{\sin^2(s)}{\abs{s}^{2+\alpha}} \dd{s} = 2^{-1- \alpha} J_{2+\alpha} = 2^{- \alpha} \frac{\sqrt{\pi}  }{1+\alpha }\frac{\Gamma((1-\alpha)/2)}{\Gamma((2+\alpha)/2)}
\end{equation}
and
\begin{equation}
    \int_{\mathbb{R}}\frac{1+\cos(t)}{(t^2+c^2)^{\frac{2+\alpha}{2}}}\dd{t} = F_{2+\alpha}(0)+F_{2+\alpha}(1) = \frac{\sqrt{\pi}}{c^{1+\alpha}}\left(\Gamma\left(\frac{1+\alpha}{2}\right)+\Psi_{\frac{1+\alpha}{2}}(c)\right).
\end{equation}
Hence, by \eqref{eq_defining_c},
\begin{equation}
    2^{- \alpha} \frac{\sqrt{\pi}  }{1+\alpha }\frac{\Gamma((1-\alpha)/2)}{\Gamma((2+\alpha)/2)} = \frac{\sqrt{\pi}}{c^{1+\alpha}}\left(\Gamma\left(\frac{1+\alpha}{2}\right)+\Psi_{\frac{1+\alpha}{2}}(c)\right),
\end{equation}
and thus,
\begin{equation}\label{eq:asymptotic_c_alpha_1}
    \lim_{\alpha \uparrow 1}c^{1+\alpha} \Gamma\left(\frac{1-\alpha}{2}\right) = \lim_{ \alpha \to 1^- }2^{\alpha }(1+\alpha)\Gamma\left(\frac{2+\alpha}{2}\right)\left(\Gamma\left(\frac{1+\alpha}{2}\right)+\Psi_{\frac{1+\alpha}{2}}(c)\right) = 4\sqrt{\pi}.
\end{equation}

Therefore, taking the limit $\alpha\uparrow 1$ in \eqref{expression_ddot_mu_hom_kernel}, we finally obtain
\begin{align}
\lim_{\alpha \uparrow 1} {\ddot{\mathcal{R}}_1(0)} &= \lim_{\alpha \uparrow 1}
\frac{1}{c^{3+\alpha}}\left\{ 
-\frac{3}{8}c^2c^{1+\alpha}\Gamma\left(\frac{1-\alpha}{2}\right) +8-32
 -\frac{32}{c^{1+\alpha}\Gamma\left(\frac{1-\alpha}{2}\right)-2}\right\}\\ 
&= \lim_{\alpha \uparrow 1} \frac{-1}{c^{3+\alpha}}
 \left\{ 24+\frac{16}{2\sqrt{\pi}-1}\right\}  = - \infty.
\end{align}
In fact, the previous computations, together with the fact that $\lim_{z\to 0} z\Gamma(z) = 1$, show that
\begin{equation}
    \lim_{\alpha \uparrow 1} (1-\alpha)^2 \ddot{\mathcal{R}}_1(0) = \frac{-1}{4\pi}\left\{ 24+\frac{16}{2\sqrt{\pi}-1}\right\} = \frac{-2}{\pi}\left\{ 3+\frac{2}{2\sqrt{\pi}-1}\right\}<0.
\end{equation}

From this, we conclude that $\ddot{\mathcal{R}}_1(0)<0 $ for $\alpha$ close to $1$, and, as a consequence, $\mathcal{R}_1(a)<0$ for $\abs{a}>0$ small enough (but nonzero), depending only on $\alpha$ and $\beta$.

\end{proof}
\section*{Acknowledgements}
The author thanks Joan de Solà-Morales and Xavier Cabré for their guidance and useful discussions on the topic of this paper.

\end{document}